\documentclass[12pt,a4paper]{amsart}

\usepackage{amsmath,amssymb,amsthm,cite, graphicx,pdfpages}

\calclayout

\newtheorem{theorem}{Theorem}

\newtheorem{corollary}[theorem]{Corollary}
\newtheorem{lemma}[theorem]{Lemma}
\newtheorem{proposition}[theorem]{Proposition}

\theoremstyle{definition}
\newtheorem{definition}[theorem]{Definition}
\newtheorem{example}[theorem]{Example}

\DeclareMathOperator*{\aff}{aff}
\DeclareMathOperator*{\Argmin}{Arg\,min}

\DeclareMathOperator*{\conv}{co}

\DeclareMathOperator*{\lspan}{span}

\DeclareMathOperator*{\relint}{ri}

\newcommand{\R}{\mathbb{R}}
\newcommand{\N}{\mathcal{N}}
\renewcommand{\P}{\mathcal{P}}

\title{Uniqueness of solutions in multivariate Chebyshev approximation problems}

\author[V. Roshchina]{Vera Roshchina}
\address[V. Roshchina]{UNSW Sydney}
\email{v.roshchina@unsw.edu.au}

\author[N. Sukhorukova]{Nadia Sukhorukova}
\address[N. Sukhorukova]{Swinburne University of Technology}
\email{nsukhorukova@swin.edu.au}

\author[J. Ugon]{Julien Ugon}
\address[J. Ugon]{Deakin University}
\email{julien.ugon@deakin.edu.au}



\begin{document}

\begin{abstract} We study the solution set to multivariate Chebyshev approximation problem, focussing on the ill-posed case when the uniqueness of solutions can not be established via strict polynomial separation. We obtain an upper bound on the dimension of the solution set and show that nonuniqueness is generic for the ill-posed problems on discrete domains. Moreover, given a prescribed set of points of minimal and maximal deviation we construct a function for which the dimension of the set of best approximating polynomials is maximal for any choice of domain. We also present several examples that illustrate the aforementioned phenomena,  demonstrate practical application of our results and propose a number of open questions.
\end{abstract}

\maketitle

\section{Introduction}

The classical Chebyshev approximation problem is to construct a polynomial of a given degree that has the smallest possible absolute deviation from some continuous function on a given interval.
 For univariate polynomials of degree $d\geq 0$ the solution is unique and satisfies an elegant alternation condition: there exist $d+2$ points of alternating minimal and maximal deviation of the function from approximating polynomial \cite{Cheb} (see Fig.~\ref{fig:alternation}).
\begin{figure}[ht]{\centering
\includegraphics[width=0.45\textwidth]{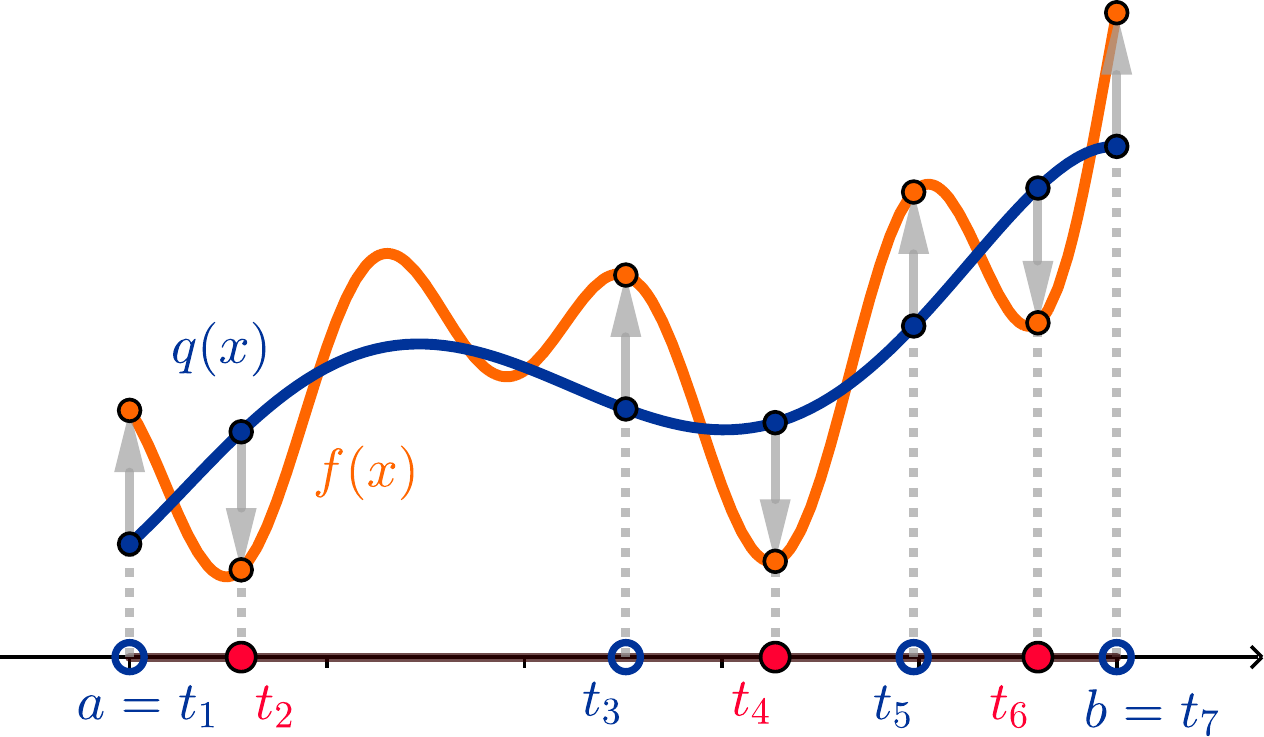}\\}
\caption{A typical distribution of the points of minimal and maximal deviation of  a continuous function ($f$, shown in blue) from its best Chebyshev approximation by a polynomial of degree at most 5 (denoted by $q$, shown in orange) on a bounded interval $[a,b]$.}
\label{fig:alternation}
\end{figure}

Once we depart from the classical case and consider approximating a continuous function on a compact subset $X$ of $\R^n$ by multivariate polynomials, the uniqueness is lost: the result of Mairhuber \cite{Mairhuber} demonstrates that a multivariate Chebyshev approximation problem has a unique solution generically (for all continuous functions on a given compact subset of $\R^n$) if and only if the underlying set $X$ is homeomorphic to a closed subset of a circle. 
In particular, 
if $X\subset\R^n$ contains an interior point, then there is no Haar space of dimension $n \geq 2$ for $X$. 
 An example of such nonunique approximation is shown in  Fig.~\ref{fig:nonunique}.
\begin{figure}[ht]{\centering
\includegraphics[width=0.45\textwidth]{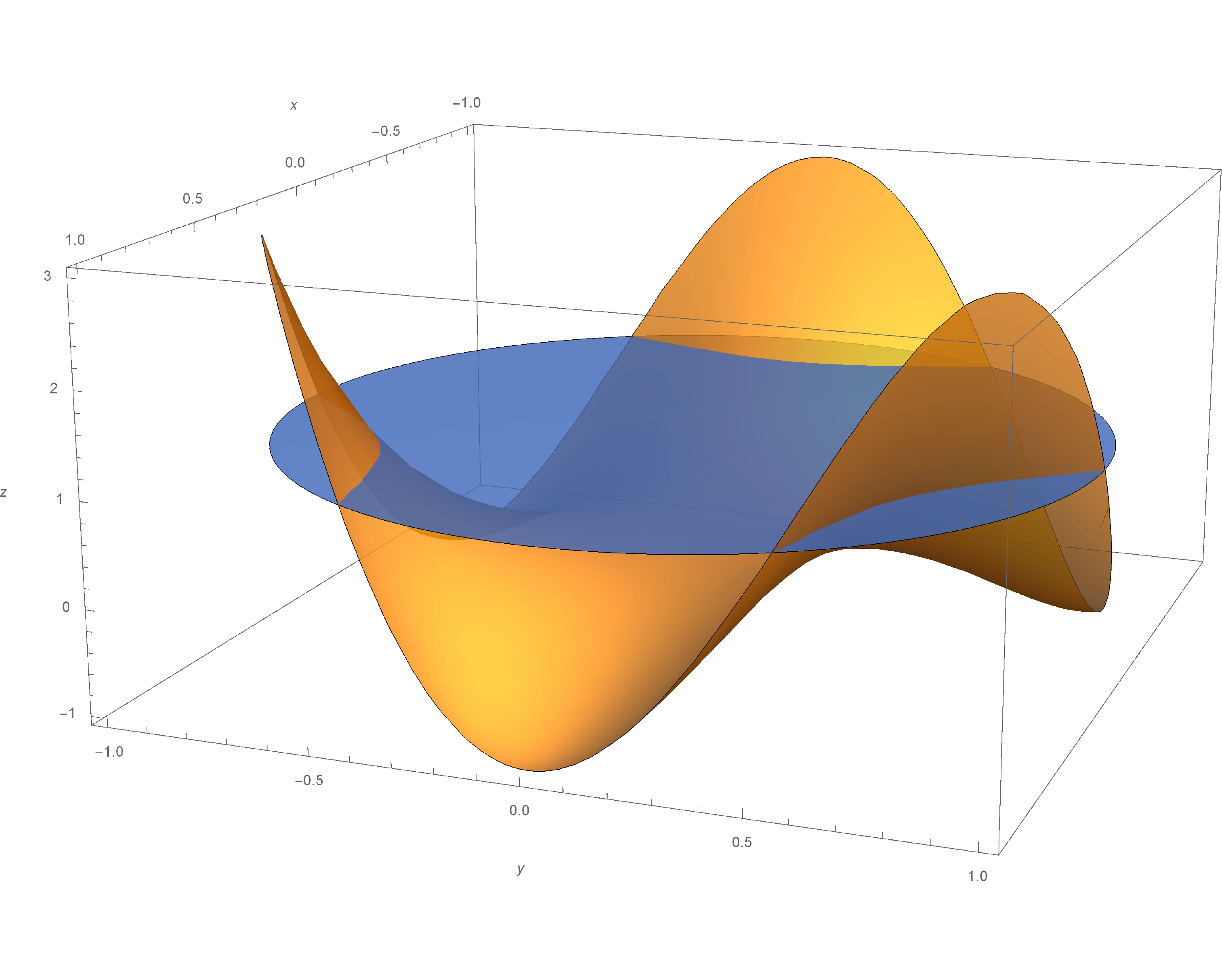}\qquad
\includegraphics[width=0.45\textwidth]{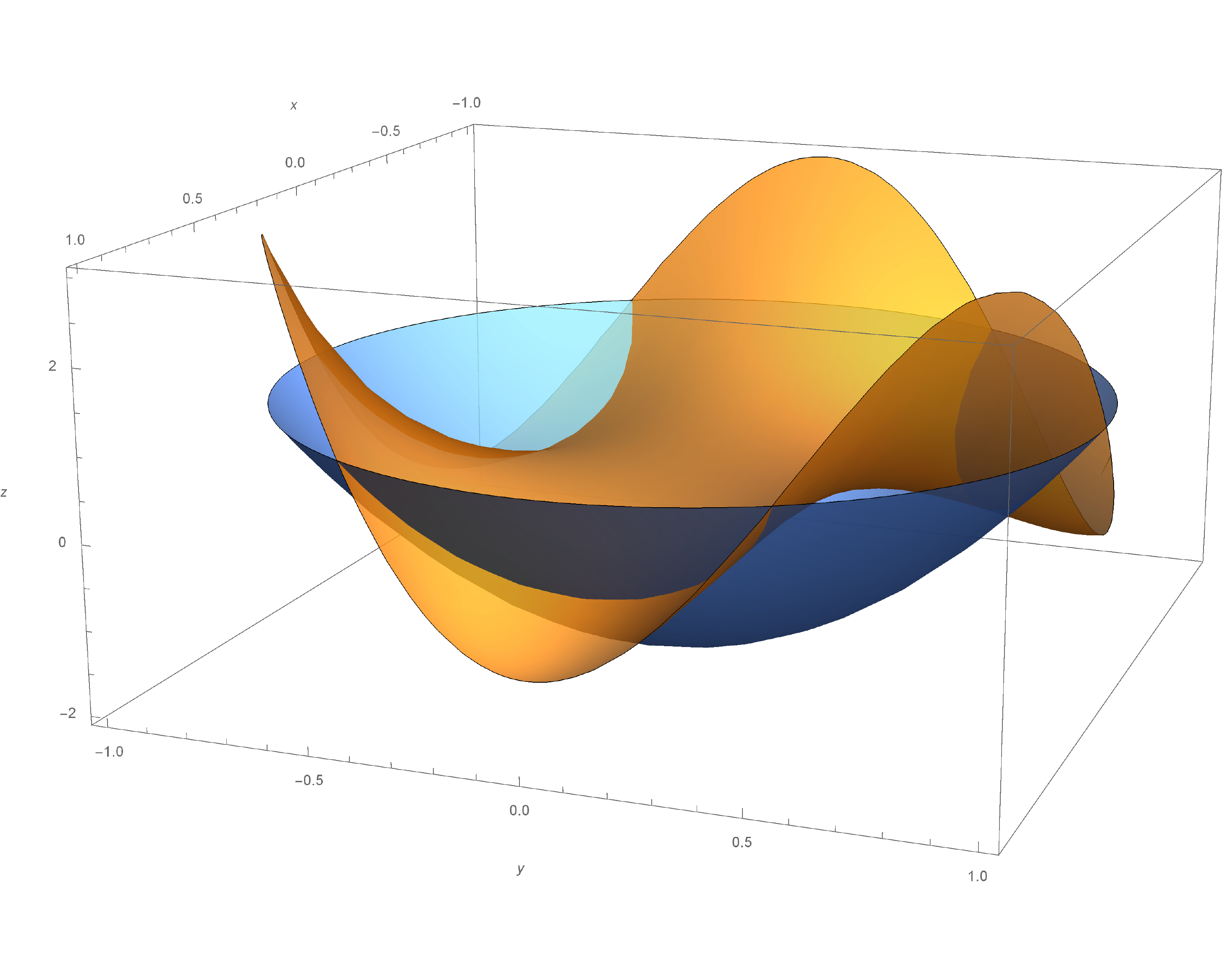}\\}
\caption{The function $f(x,y) =  x^6 +y^6+  3 x^4 y^2 + 3 x^2 y^4 + 6 x y^2 - 2 x^3$ has several best quadratic approximations on the disk $x^2+y^2\leq 1$. The plot of the function in orange colour is shown together with two different best approximations in blue: $q_0(x,y) = 1$ (on the left) and $q_1(x,y) =  3 x^2 + 3 y^2-2$ (on the right). }
\label{fig:nonunique}
\end{figure}

Even though the uniqueness of solutions is lost in the multivariate case, the alternation result holds in the form of algebraic separation. It was first shown in \cite{Rice} that a polynomial approximation of degree $d$ is optimal if the sets of points of minimal and maximal deviation can not be separated by a polynomial of degree at most $d$. This result can be reproduced using the standard tools of modern convex analysis, as demonstrated in \cite{ModernAlt}. Another approach to generalise the notion of alternation to multivariate problems is based on the alternating signs of certain determinants~\cite{DemyanovMalozemov}.

The classical alternation result was obtained by Chebyshev in 1854 \cite{Cheb}, but little is known about the shape of the solutions of a more general multivariate problem.  In particular, related work \cite{IntCheb} that studies a version of this problem for polynomials with integer coefficients, mentions that the multivariate problem is `virtually untouched'. Even though the solutions to the multivariate problem satisfy a form of an alternation condition, the structure of the solutions and the location of points of maximal and minimal deviation are more complex compared to the univariate case, which results in many interesting challenges.

From the point of view of classical approximation theory multivariate polynomial approximation is relatively inefficient: for a range of key applications some other approaches such as the radial basis functions \cite{radialbasisbook} provide superior results. However modern optimisation is increasingly fusing with computational algebraic geometry, successfully tackling problems that were insurmountable in the past, and polynomial approximation emerges in this context as valuable not only for solving computationally challenging problems, but also as an analytic tool that together with Gr\"obner basis methods may lead to algorithmic solutions for finding extrema in nonconvex problems. Another potential application is a generalisation of trust-region methods, where instead of local quadratic approximations to the function locally more versatile higher order polynomial approximations may be used. 

Consider the space $\mathbb{P}_d(\R^n)$ of real polynomials in $n$ variables of degree at most $d$. Let $f:X\to \R$ be a continuous function defined on a compact set $X\subset \R^n$. A polynomial $q^*\in \mathbb{P}_d(\R^n)$ solves the multivariate Chebyshev approximation problem for $f$ on $X$ if 
$$
\max_{x\in X}|f(x) - q^*(x)|\leq \max_{x\in X} |f(x)-q(x)|\quad \forall q\in \mathbb{P}_d(\R^n).
$$
We are interested in the set $Q\subset\mathbb{P}_d(\R^n)$ of all such solutions. In some special cases the solution to the multivariate Chebyshev approximation problem is known explicitly. For instance, the best approximation by monomials on a unit cube is obtained from the products of classical Chebyshev polynomials (see \cite{ThiranDetaille} and a more recent overview  \cite{xuyuan}); this is related to another generalisation of Chebyshev's results, when the problem of a best approximation of zero with polynomials having a fixed highest degree coefficient is considered: in some special cases, solutions on the unit cube are known from \cite{sloss};  solutions for the unit ball were obtained in \cite{Reimer}. 

There is a different approach to generalising Chebyshev polynomials, based on extending the relation 
$T_k(\cos x) = \cos k x$ to the multivariate case. In \cite{MuntheKaasChapter,MuntheKaasFoCM} more general functions $h:\R^n\to \R^n$ periodic with respect to fundamental domains of affine Weyl groups are considered, and the aforementioned relation  is replaced by $P_k(h(x)) = h(kx)$. Such generalised Chebyshev polynomials are in fact systems of polynomials, as $P_k:\R^n\to \R^n$. We note here that the aforementioned work, as well as other approximation techniques based on Chebyshev polynomials (common in numerical PDEs), use nodal interpolation with Chebyshev polynomials. This is a conceptually different framework compared to our optimisation setting; in particular, this approach requires a careful choice of interpolation nodes on the domain to ensure the quality of approximation.   


For the univariate problem the optimal solutions to the Chebyshev approximation problem can be obtained using numerical techniques that fit in the context of linear programming and the simplex method, and exchange algorithm pioneered by Remez \cite{remez} is perhaps the most well-known technique. Even though the multivariate problem can be solved approximately by linear programming, the problem rapidly becomes intractable with the increase in the degree and number of variables, and hence there is much need for more efficient methods. This is another exciting research direction, as the rich structure of the problem is likely to yield specialised methods which surpass the performance of direct linear programming discretisation. The general framework for the potential generalisation of the exchange approach was laid out in \cite{VP}, however several implementation issues need to be resolved for a practically viable version of the method.

For any polynomial $q$ we can define the sets of points of minimal and maximal deviation, i.e. such $x\in X$ for which  the values $q(x)-f(x)$ and $f(x)-q(x)$ respectively coincide with the maximum $\max_{x'\in X}|f(x')-q(x')|$. These sets may be different for different polynomials in the optimal set $Q$. We show  that it is possible to identify an intrinsic pair of such subsets pertaining to all polynomials in $Q$ (see Theorem~\ref{thm:minimalset}); moreover the location of these points determines the maximal possible  dimension of the solution set (see Lemma~\ref{lem:dim}). We also show that for any prescribed arrangement of points of minimal and maximal deviation and any choice of the maximal degree there exists a continuous function and a relevant approximating polynomial for which these points are precisely the points of minimal and maximal deviation; moreover, the set of all best approximations has the largest possible dimension, for any choice of domain $X$ (Lemma~\ref{lem:sharpbump}). Finally, we show that the set of best Chebyshev approximations is always of the maximal possible dimension if the domain $X$ is finite (Lemma~\ref{lem:finite}).

We begin with some preliminaries and examples in Section~\ref{sec:prelim}, focussing on the well-known separation characterisation of optimality and Mairhuber's uniqueness result. In Section~\ref{sec:structure} we present our new results. We then summarise our findings and present some open problems in Section~\ref{sec:concl}.

\section{Preliminaries and Examples}\label{sec:prelim}

\subsection{Multivariate polynomials}

A multivariate polynomial of degree $d$ with real coefficients can be represented as 
\[
q(x) = \sum_{|\alpha|\leq d} a_\alpha x^\alpha,
\]
where  $\alpha = (\alpha_1,\dots, \alpha_n)$ is an $n$-tuple of nonnegative integers, $x^d = x_1^{\alpha_1} x_2^{\alpha_2}\cdots x_n^{\alpha_n}$, $|\alpha| = |\alpha_1|+|\alpha_2|+\cdots+ |\alpha_n|$, and $a_\alpha\in \R$ are the coefficients. 
All polynomials of degree not exceeding $d$ constitute a vector space $\mathbb{P}_d(\R^n)= \lspan \{x^\alpha\,|\, |\alpha|\leq d\}$
of dimension ${{n+d} \choose {d}}$.

Note that, generally speaking, we can consider any finite set of (linearly independent) polynomials in $n$ variables, $G = (g_1,\dots, g_N)$ and instead of the space $\mathbb{P}_d(\R^n)$ consider the linear span $V$ of $G$, i.e.
\begin{equation}\label{eq:V}
V = \lspan \{g_i\,|\, i \in \{1,\dots, N\}\}.
\end{equation}
Then the solution set $Q\subseteq V$ to the Chebyshev approximation problem for a given continuous function $f$ defined on a compact set $X\subseteq \R^n$ is 
\begin{equation}\label{eq:p}
Q: = \Argmin_{q\in V}\|f-q\|_\infty,
\end{equation}
where
$$
\|f-q\|_{\infty} = \max_{x\in X}|f(x) - q(x)|.
$$

Fixing a continuous function $f:X\to \R$, for every polynomial $q\in V$ we define the sets of points of minimal and maximal deviation explicitly as
\begin{align}\label{eq:defNP}
\N(q)& := \{x\in X\,|\,  q(x)-f(x) = \|f-q\|_\infty\},\notag\\
\P(q)& := \{x\in X\,|\,  f(x)-q(x) = \|f-q\|_\infty\}.
\end{align}
Observe that for any given polynomial $q$ at least one of these sets is nonempty, and for any $q^*\in Q$ both of them are nonempty (otherwise one can add an appropriate small constant to $q^*$ and decrease the value of the maximal absolute deviation). Also observe that the sets $\N(q)$ and $\P(q)$ are disjoint unless $q\equiv f$ on $X$ (in this case $\N(q) = \P(q) = X$).

The minimisation problem of \eqref{eq:p} is an unconstrained convex optimisation problem: the objective function $\|f-q\|_\infty$ can be interpreted as the maximum over two families of linear functions parametrised by the domain variable $x\in X$, i.e.
\begin{equation}\label{eq:max}
\|f-q\|_\infty  = \max_{x\in X}|f(x)-q(x)|= \max_{\substack{x\in X\\ s}\in \{-1,1\}}s(f(x)-q(x)).
\end{equation}
The solution set $Q$ is nonempty, since it represents the metric projection of $f$ onto a finite-dimensional linear subspace $V$ of the normed linear space of functions bounded on $X$. It is also easy to see from the continuity of $f$ that this set is closed. Moreover, since a maximum function over a family of linear functions is convex, $Q$ is convex (e.g. see \cite[Proposition~{2.1.2}]{JBHU}).

\begin{example}[Solution set is unbounded] We consider a degenerate case of the problem: find the best linear approximation to $f(x,y) = x^2$ on $X=[-1,1]\times\{0\}$. Since the domain is effectively restricted to the line segment $[-1,1]$, the solution reduces to the classical univariate case: there is a unique best approximation, which happens to be constant, $\frac{1}{2}$. Observe however that in the true two-dimensional setting any linear polynomial of the form $q(x,y) = \frac{1}{2}+\alpha y$ is also a best approximation of $f$ on $X$. This means that the solution set of best approximations is unbounded, $Q = \{\frac{1}{2}+ \alpha y, \, \alpha \in \R\}$, even though all such optimal solutions coincide on $X$, and effectively---on the set $X$---provide the same unique best approximation. 
\end{example}

\subsection{Optimality conditions}
\begin{definition}
We say that a polynomial $p\in V$ separates two sets $N,P\subset \R^n$ if 
\begin{equation}\label{eq:separation}
p(x)\cdot p(y) \leq 0 \quad \forall x\in N, y\in P;
\end{equation}
we say that the separation is \emph{strict} if the inequality in \eqref{eq:separation} is strict, i.e.
\begin{equation}\label{eq:strictseparation}
p(x)\cdot p(y) < 0 \quad \forall x\in N, y\in P.
\end{equation}
\end{definition}
Recall the well-known characterisations of optimality (see \cite{Rice} and \cite{ModernAlt} for modern proofs).

\begin{theorem}\label{thm:classicsep} Let $X$ be a compact subset of $\R^n$, and assume that $f: X\to \R$ is a continuous function. A polynomial $q\in V$ is an optimal solution to the Chebyshev approximation problem \eqref{eq:p} if and only if there exists no  $p\in V$ that strictly separates the sets $\N(q)$ and $\P(q)$.
\end{theorem}

\begin{example}[Best quadratic approximation is not unique]\label{eg:nonunique} We focus on the function  $f(x,y) =  x^6 +y^6+  3 x^4 y^2 + 3 x^2 y^4 + 6 x y^2 - 2 x^3$ discussed in the Introduction and demonstrate that it does indeed have multiple best quadratic approximations on the disk $x^2+y^2\leq 1$ (see Fig.~\ref{fig:nonunique}).

For two different polynomials $q_0(x,y) = 1$ and $q_1(x,y) =  3 x^2 + 3 y^2-2$ the points of maximal negative and positive deviation of $f$ from these polynomials are
\[
\N(q_0) = \{z_1,z_3,z_5\}, \quad \N(q_1) = \N(q_0)\cup\{z_0\}, \quad \P(q_0) = \P(q_1) = \{z_2,z_4,z_6\}, 
\]
where 
\[
z_0 = (0,0),\; z_1 =  (1,0),\; z_2 = \left(\frac{1}{2},\frac{\sqrt{3}}{2}\right),\;  z_3 = \left(-\frac{1}{2},\frac{\sqrt{3}}{2}\right),
\]
\begin{equation}\label{eq:defz}
 z_4 = (-1,0), \; z_5 =  \left(-\frac{1}{2},-\frac{\sqrt{3}}{2}\right),\; z_6 = \left(\frac{1}{2},-\frac{\sqrt{3}}{2}\right).
\end{equation}
This is not difficult to verify using standard calculus techniques (see appendix).

\subsection{Location of maximal and minimal deviation points}

Observe that the points $z_1,z_2,\dots, z_6$ lie on the unit circle. By the B\'ezout theorem, this circle can have at most 4 intersections with any other quadratic curve. However if we could find  a quadratic polynomial that strictly separates the points of maximal and minimal deviation, the relevant curve would intersect the circle in at most six points, as shown in Fig.~\ref{fig:circle-curve}.
\begin{figure}[ht]{\centering
\includegraphics[width=0.3\textwidth]{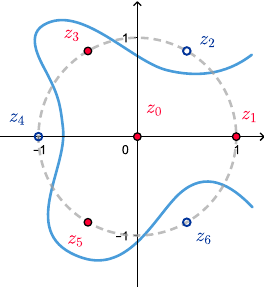}\qquad
\includegraphics[width=0.3\textwidth]{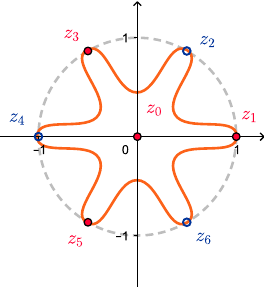}\\}
\caption{On the left: the intersection of two quadratic curves at six points contradicts the B\'ezout theorem; on the right: a subset of the unit disk homeomorphic to a circle.}
\label{fig:circle-curve}
\end{figure} Hence such separation is impossible, so both $q_0$ and $q_1$ are optimal. 
\end{example}

We conclude this section with the well-known result of Mairhuber \cite{Mairhuber} (generalised to compact Hausdorff spaces by Brown~\cite{brown}).

\begin{theorem}[Mairhuber] 
A compact subset $X$ of $\R^n$ containing at least $k \geq 2$ points may serve as the domain of definition of a set of real continuous functions $f_1(x),\dots,f_k(x)$ that provide a unique Chebyshev approximation to any continuous function $f$ on the set $X$, if and only if $X$ is homeomorphic to a closed subset of the circumference of a circle.
\end{theorem}

With relation to our setting, Mairhuber's result is effectively a necessary condition for generic uniqueness, since our choice of the system of functions is restricted to multivariate polynomials. Hence it is possible to identify a compact set $X$ homeomorphic to a circle and a set of polynomials linearly independent on $X$ that do not provide a unique multivariate approximation to a continuous function on $X$.

\begin{example} Observe that any best approximation to $f$ from Example~\ref{eg:nonunique} on the disk is also the best approximation to $f$ on any subset of the disk that contains the sets $\N(q_0)$ and  $\P(q_0)$. Even though the two different best approximations $q_0$ and $q_1$ coincide on the boundary of the disk, they take different values everywhere in the interior, and hence we can choose another subset of the unit disk that is homeomorphic to a circle (like the one shown in  Fig.~\ref{fig:circle-curve} on the right) to obtain two different optimal solutions. This does not contradict Mairhuber's theorem, since in this case we have restricted ourselves to a very specific choice of the basic functions. 
\end{example}

\section{Structure of the solution set}\label{sec:structure}
\subsection{The location of maximal and minimal deviation points for different optimal solutions}

The key technical result of this section is the following theorem that establishes the existence of uniquely defined subsets of points of maximal and minimal deviation across all optimal solutions. This means that the points of maximal and minimal deviation do not wander around the domain $X$ as we move from one optimal solution to another. 

\begin{theorem}\label{thm:minimalset} Let $f:X\to \R$ be a continuous function defined on a compact set $X\subset\R^n$, let $V$ be a subspace of multivariate polynomials in $n$ variables \eqref{eq:V}, and suppose that $Q$ is the set of optimal solutions to the relevant optimisation problem, as in \eqref{eq:p}. Then
\begin{itemize}
	\item[(i)] $\N(q) = \N(p)$, $\P(q) = \P(p)$ $\forall p,q\in \relint Q$;
	\item[(ii)] $\N(q) \subseteq \N(p)$, $\P(q) \subseteq \P(p)$ $\forall q\in \relint Q, p \in Q$.
\end{itemize}
Here the relative interior is considered with respect to the convex sets of the coefficients in the representation of the solutions as linear combinations of polynomials in $V$.
\end{theorem} 

For the proof of this lemma, we will need the following elementary result about max-type convex functions.

\begin{proposition}\label{prop:linmax} Let $f:\R^n\to \R$ be a pointwise maximum over a family of linear functions, 
$$
f(x) = \max_{t\in T} f_t(x), \quad f_t:\R^n\to \R \text{ linear } \forall t\in T.
$$
Let $I(x) = \{t\,|\, f_t(x) = f(x)\} $,  $Q:= \Argmin\limits_{x\in \R^n}f(x)$. If $Q\neq \emptyset$, then 
\[
I(x) \subseteq I(y)\quad \forall x\in \relint Q, y\in Q.
\]
\end{proposition}
\begin{proof} Let $x\in \relint Q$, $y\in Q$. Assume that there exists $t\in T$ such that $t\in I(x)\setminus I(y)$. Then $f(x) = f(y) = f_t(x)>f_t(y)$, and since $f_t$ is linear, we then have
$$
f(x-\alpha (y-x) ) \geq f_t(x-\alpha (y-x)) = f_t(x) - \alpha (f_t(y)-f_t(x)) > f(x) \quad \forall \alpha>0,
$$ 
hence, $x-\alpha (y-x)\notin Q$ for $\alpha >0$, while $y=x+(y-x)\in Q$, which means $x\notin \relint Q$, a contradiction. 
\end{proof}

\begin{proof}[Proof of Theorem~\ref{thm:minimalset}] Recall that our objective function can be represented as the maximum over a family of linear functions, as in \eqref{eq:max}. For every polynomial $q\in V$ define the set of active indices 
\[
I(q) = \{(x,s)\in X\times\{-1,1\}\,|\, s (f(x)-q(x)) = \|f-q\|_\infty\}.
\]
It is evident from the definition  \eqref{eq:defNP} of $\N(q)$ and $\P(q)$ that 
$$
x\in \N(q) \; \Leftrightarrow (x,-1) \in I(q); \quad x\in \P(q) \; \Leftrightarrow (x,1) \in I(q).
$$
The result now follows from Proposition~\ref{prop:linmax}.
\end{proof}

The following 
corollary of Theorem~\ref{thm:minimalset} characterises the structure of the location of maximal deviation points corresponding to different optimal solutions. 

\begin{corollary}
The sets of points of minimal and maximal deviation remain constant if the optimal solutions belong to  the relative interior of the solution set. Additional maximal and minimal deviation points  can only occur if an optimal solution is on the relative boundary. 
\end{corollary} 

For any given continuous function $f$ defined on a compact set $X$ we can hence define the \emph{minimal} or \emph{essential} sets of points of minimal and maximal deviation, 
$$
\P = \P(q), \; \N = \N(q), \; q\in \relint Q,
$$
where $\P(q)$ and $\N(q)$ are defined in the standard way, as in \eqref{eq:defNP}. For instance, in Example~\ref{eg:nonunique} we have $\N = \{z_1,z_3,z_5\}$ and $\P = \{z_2,z_4,z_6\}$, while $\N(q_1)$ contains an additional point $z_0$. 

Note that the essential pair of sets is uniquely defined, and is different to the definition of critical subsets given in \cite{Rice}.

\subsection{Dimension of the solution set}
We next focus on the relation between the family of separating polynomials and the dimension of solution set. 

For a fixed continuous function $f:X\to \R$ and a polynomial $q\in V$ consider the set of all polynomials in $V$ that separate the points of minimal and maximal deviation,
$$
S(q) = \{s\in V\,|\, s(x)\cdot s(y)\leq 0 \, \forall x\in \P(q),y\in \N(q)\}. 
$$
Notice that the zero polynomial is always in $S(q)$, and for the polynomials in the optimal solution set we may have a nontrivial set of separating functions. This happens in particular when all points of minimal and maximal deviation are located on an algebraic variety of a subset of $V$.

Since the pair of sets of minimal and maximal deviation is minimal on the interior of $Q$, and such minimal pair is unique according to Theorem~\ref{thm:minimalset}, we can define the \emph{maximal} set of separating polynomials as $S = S(q) $ for $q\in \relint Q$.

For the rest of the section, we work with an arbitrary fixed continuous real-valued function $f$ defined on a compact set $X\subset \R^n$, so we do not repeat this assumption in each statement, and simply refer to the solution set $Q$ of the corresponding Chebyshev approximation problem.

\begin{lemma}\label{lem:dim} For the solution set $Q$ we have $\dim Q\leq \dim S$; moreover, for any $q,p\in Q$ we have 
	$p-q\in S(p)\subseteq S$.     
\end{lemma}
\begin{proof} Observe that it is enough to show that for any $q\in \relint Q$ and any $p\in Q$ we have $p-q\in S$. It then follows that $\aff Q \subseteq S+q$, and hence $\dim Q \leq \dim S$. 

Let $q\in \relint Q$ and assume $p\in Q$. By Theorem~\ref{thm:minimalset} we have $\N(q) \subseteq \N(p)$, $\P(q)\subseteq \P(p)$, therefore
\begin{align*}
f(u) - p(u) &\leq f(u) -q(u)\quad  \forall u \in \N(p), \\
f(u) - p(u) &\geq f(u) -q(u)\quad \forall u \in \P(p).
\end{align*}
Let $s(x) = p(x)-q(x)$. We have 
\[
s(u) =  p(u) - q(u) \geq 0 \quad \forall u\in \N(p), \qquad
s(u) = p(u)- q(u) \leq 0 \quad \forall u\in \P(p),
\]
and so $s(u) \in S(p)\subseteq S(q)$. 
\end{proof}

\begin{corollary}
If for the solution set $Q$ we have $\dim Q >0$, then all essential points of minimal and maximal deviation lie \emph{on} a variety of some nontrivial polynomial $s\in V$. 
\end{corollary}
\begin{proof} This follows directly from a modification of the proof of Lemma~\ref{lem:dim}: if $Q$ is of dimension 1 or higher, then there exist two different polynomials $q\in \relint Q$ and $p\in Q$. We have  
\begin{align*}
f(u) - p(u) &= f(u) -q(u)\quad  \forall u \in \N = \N(q), \\
f(u) - p(u) &= f(u) -q(u)\quad \forall u \in \P= \P(q).
\end{align*}
Hence, for $s(x)  =p(x)-q(x)$ we have $s(u) = p(u)-q(u) = 0 \quad \forall u \in \N\cup \P$. 
\end{proof}

The next corollary is a well-known uniqueness result. 

\begin{corollary}\label{cor:unique} If the set $S$ is trivial, then the optimal solution is unique.
\end{corollary}
\begin{proof} If $S=\{0\}$, then $\dim S = 0$, and by Lemma~\ref{lem:dim} we have $\dim Q = 0$. 
\end{proof}	
\subsection{Uniqueness and smoothness}
It may happen that the dimensions of $Q$ and $S$ do not coincide, as we demonstrate in the next example.

\begin{example}\label{eg:unique-and-not} Let $f(x,y) = (x^2 - \frac{1}{2})(1 - y^2)$ and consider the problem of finding a best linear approximation of this function on the square $X = [-1,1]\times[-1,1]$. 
\begin{figure}[ht]{\centering
\includegraphics[width=0.3\textwidth]{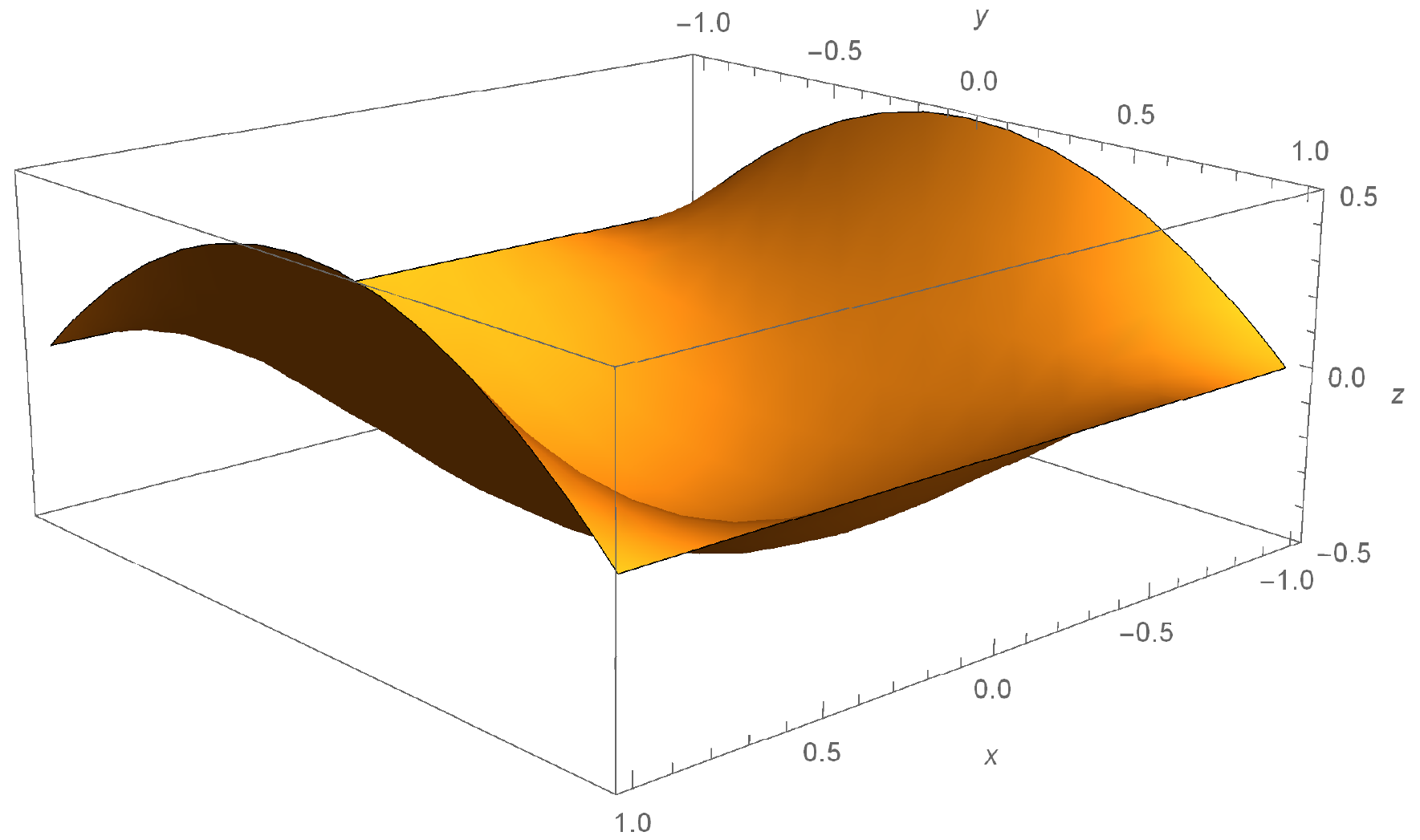}\quad 
\includegraphics[width=0.3\textwidth]{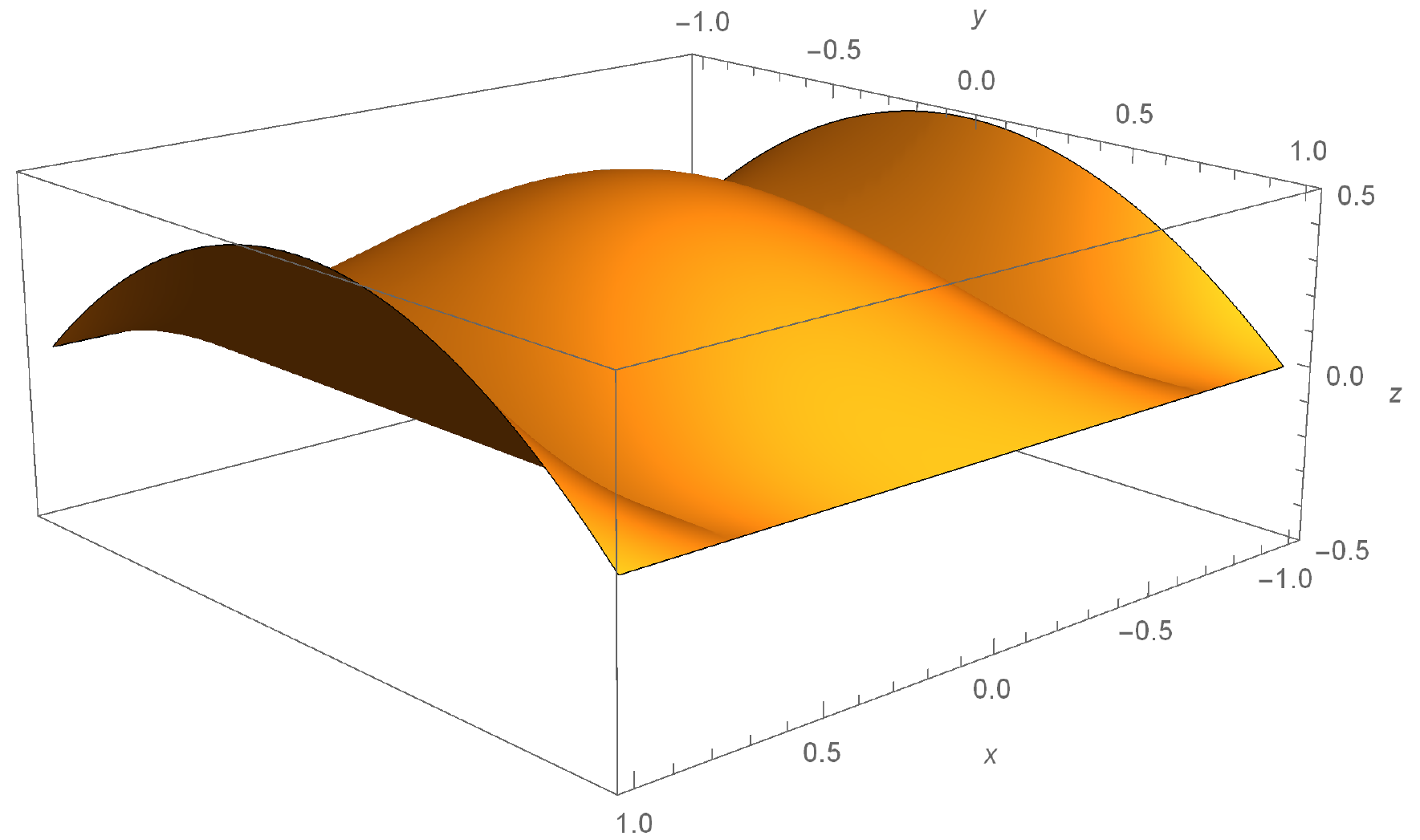}\quad 
\includegraphics[width=0.3\textwidth]{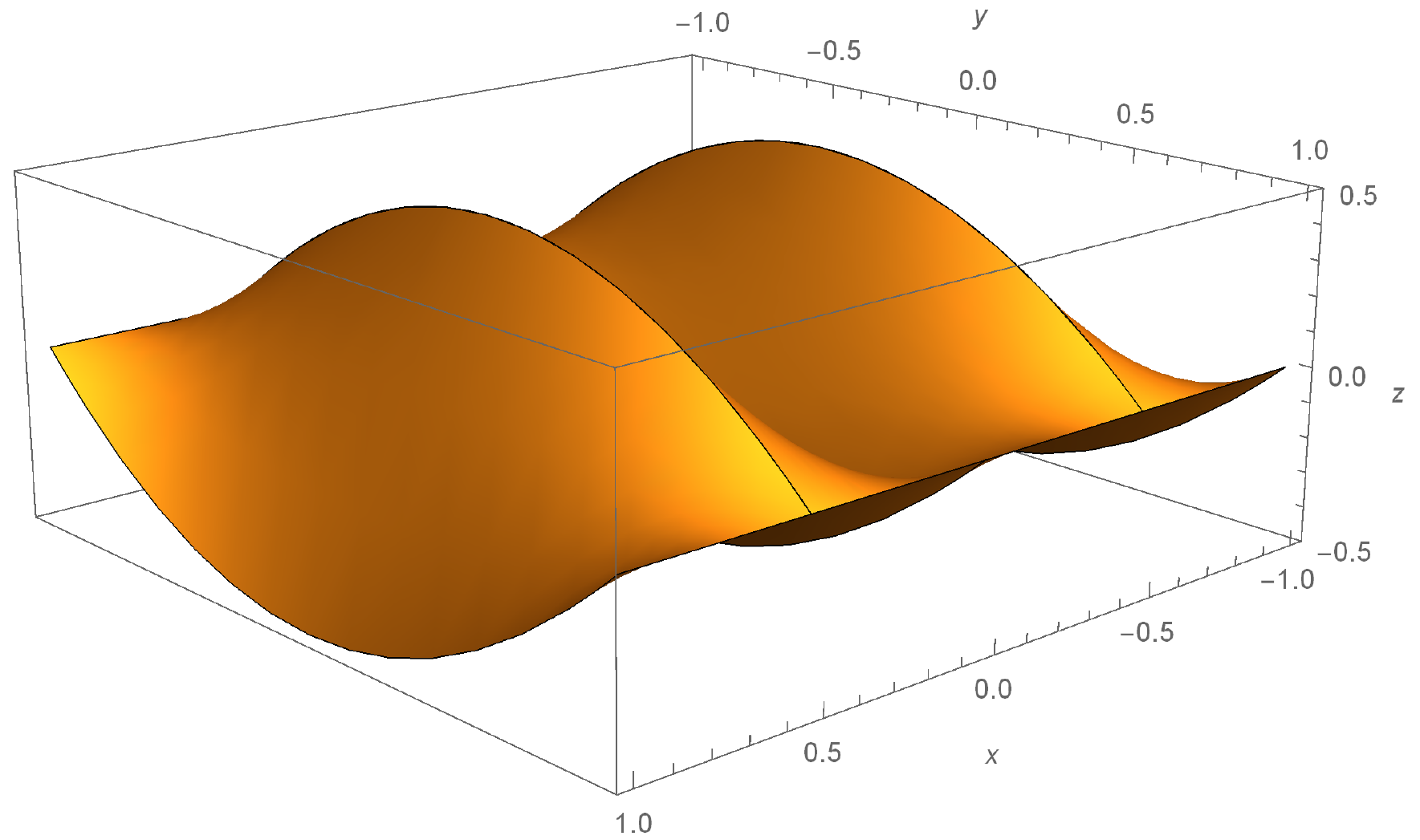}\\}
\caption{The function $f(x,y) = (x^2 - \frac{1}{2})(1 - y^2)$ (on the left), the absolute deviation of $f$ from the constant $q_0(x,y) \equiv 0$, $|d_0(x,y)| = |f(x,y)-q_0(x,y)|$ (middle), and the function $g(x,y)$.}
\label{fig:function1}
\end{figure}
It is not difficult to verify that the constant function $q_0(x,y)\equiv 0$ is an optimal solution: the points of maximal deviation are the maxima of $f(x,y)$ on the square,  attained at $\P(q_0) = \{(1,0),(-1,0)\}$; the set of points of minimal deviation is a singleton $\N(q_0)  = \{(0,0)\}$ (we provide technical details in the appendix). 

Since these three alternating points of maximal and minimal deviation lie on a straight line $y=0$, there is no strict linear separator between them (see the left image in Fig.~\ref{fig:function2dev}), hence this constant solution must be optimal by Theorem~\ref{thm:classicsep}. Also notice that taking any point out of either $\N(q_0)$ or $\P(q_0)$ ruins the optimality condition (in fact, our configuration of the points of minimal and maximal deviation is \emph{critical} in the notation of \cite{Rice}). Hence we must have $\N = \N(q_0)$ and $\P = \P(q_0)$, so these are the essential sets of the points of minimal and maximal deviation. These three points can be separated non-strictly by the linear functions of the form $l(x,y)= \alpha y$, $\alpha \in\R$. We therefore have
\[
S  =  \{\alpha y \,|\,  \alpha \in \R\}.
\] 
Even though $\dim S = 1$, the best linear approximation is unique. It follows from Lemma~\ref{lem:dim} that $Q\subseteq S$, and hence any best linear approximation should have the form $q_\alpha(x,y) = \alpha y$ for some $\alpha \in \R$. When $x = \pm 1$, we have the deviation $d_\alpha (x,y) = f(x,y) - q_\alpha(x,y) = \frac{1-y^2}{2}- \alpha y $. The maximun of $d_\alpha (x,y)$ is attained at $y=-\alpha$, with the value $d_\alpha(\pm 1, -\alpha) = \frac{1}{2}+\frac{\alpha^2}{2}>\frac{1}{2}$ for $\alpha\neq 0$, which means that there are no optimal solutions in the neighbourhood of $q_0(x,y)\equiv 0$, and hence, due to the convexity of $Q$, the best approximation is unique.

\begin{figure}[ht]{\centering
\includegraphics[width=0.3\textwidth]{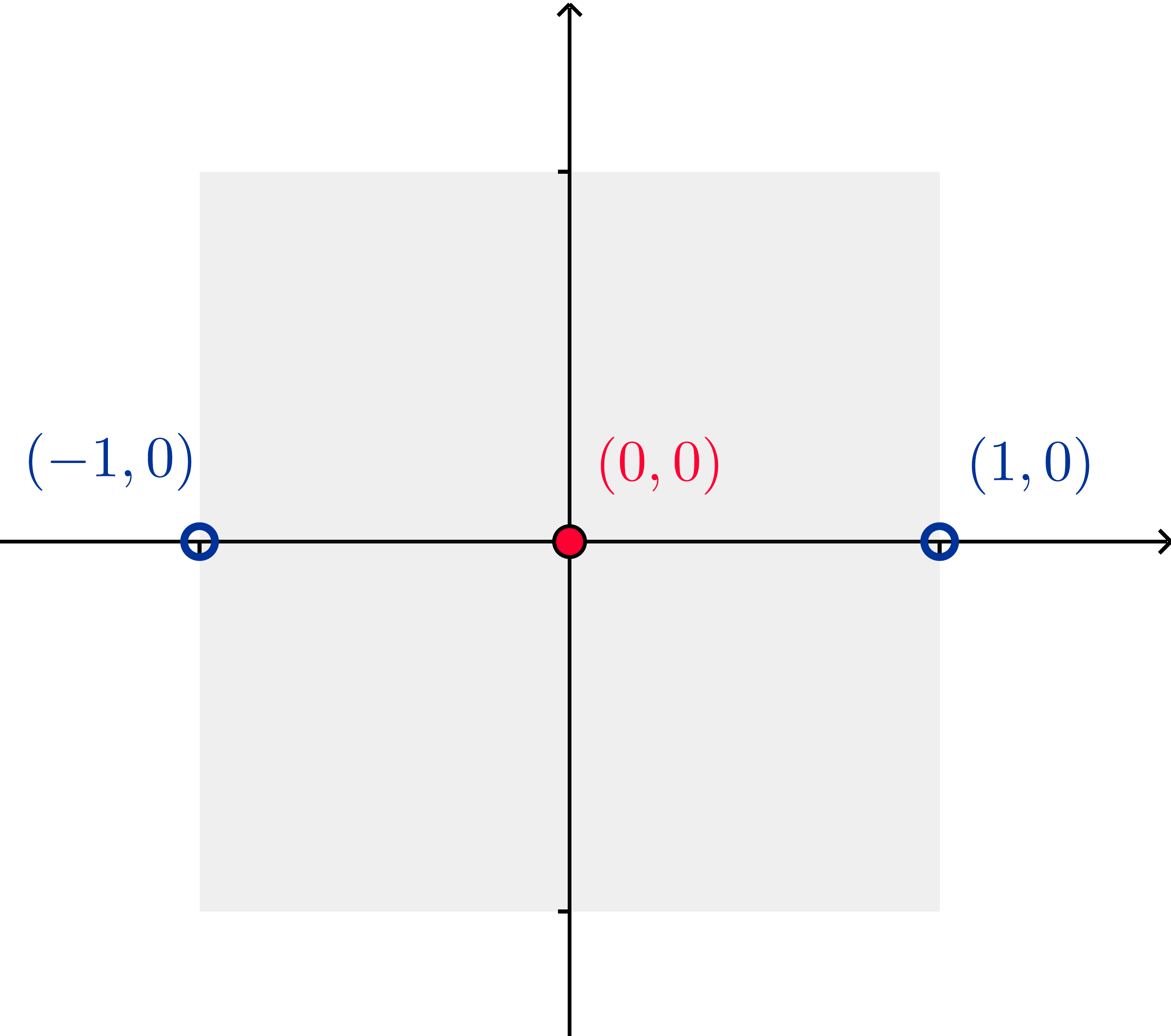}\quad 
\includegraphics[width=0.3\textwidth]{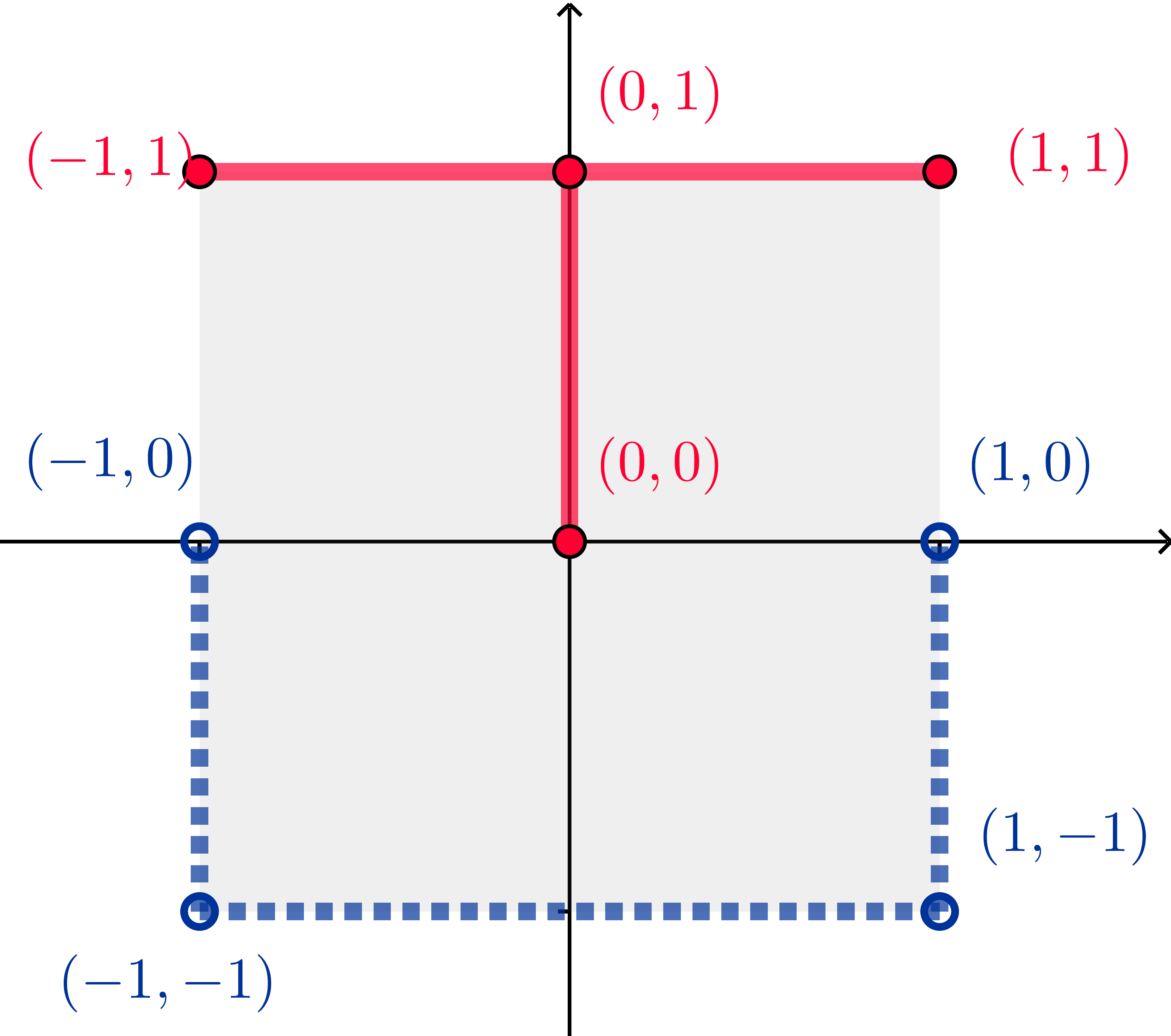}\quad 
\includegraphics[width=0.3\textwidth]{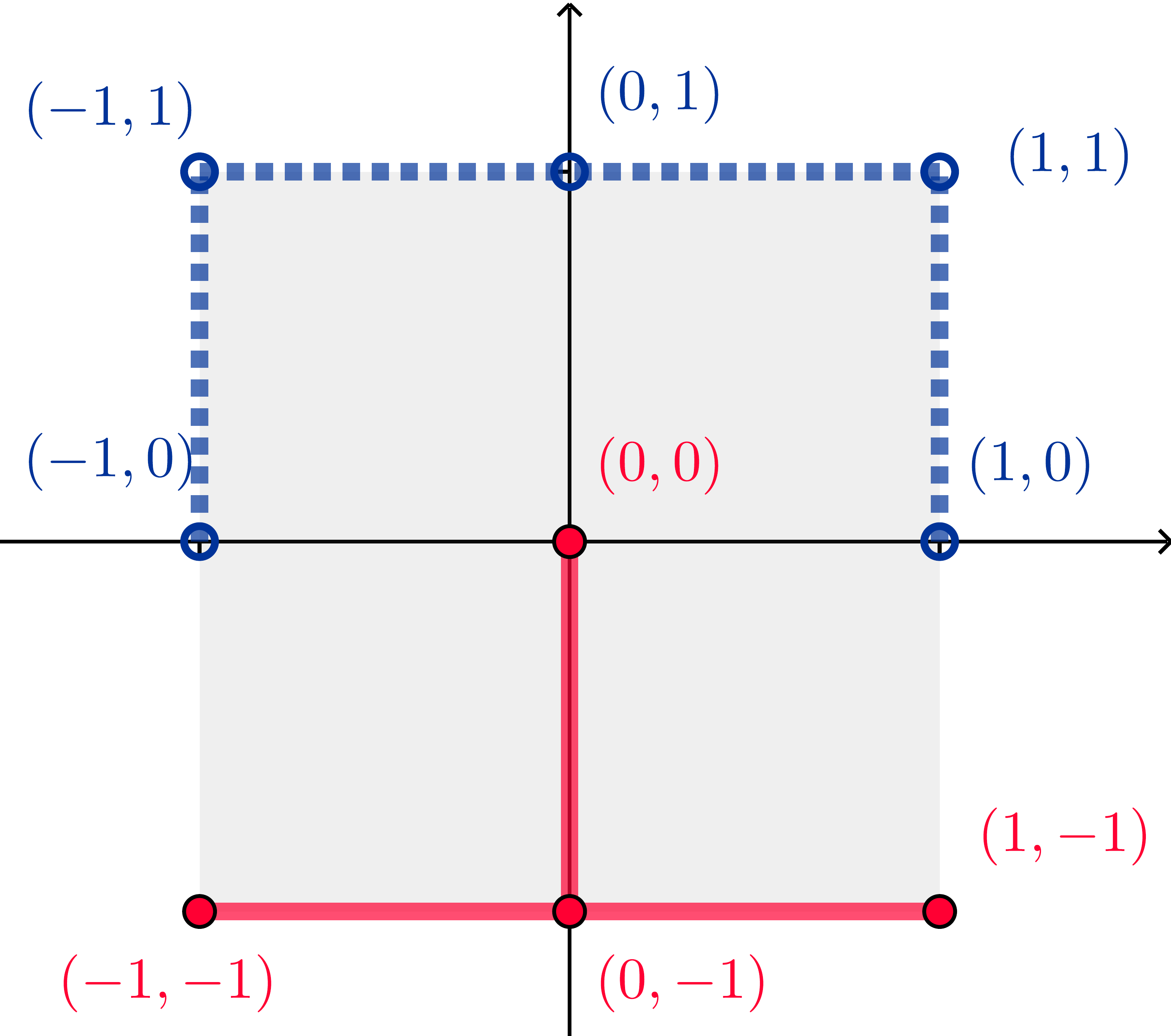}\\}
\caption{The points of minimal and maximal deviation for different cases: on the left for $f$ and $h$ and $q_0$; in the middle for $h$ and $q_{\frac{1}{2}}(x,y) = \frac{y}{2}$; on the right for $h$ and $q_{-\frac{1}{2}}(x,y) = -\frac{y}{2}$.}
\label{fig:function2dev}
\end{figure}

Now consider a modified example: let $h(x,y) = (x^2 - \frac{1}{2})(1 - |y|)$ (see Fig.~\ref{fig:function2}, left hand side). The same trivial constant function $q_0(x,y)\equiv 0$ is a best linear approximation to $h$, with the same sets of points of minimal and maximal deviation (see Fig.~\ref{fig:function1}, right). However, this best approximation is not unique: any function $q_\alpha(x,y)=\alpha y $  for $\alpha \in \left[-\frac{1}{2},\frac{1}{2}\right]$ is also a best linear approximation of $f$ on the square $X$ (see appendix for technical computations). Moreover, the sets of points of maximal and minimal deviation are different at the endpoints of the optimal interval, i.e. for $\alpha = \pm \frac 12$, see Fig.~\ref{fig:function2dev} (the technical computations are presented in appendix).

\begin{figure}[ht]{\centering
\includegraphics[width=0.32\textwidth]{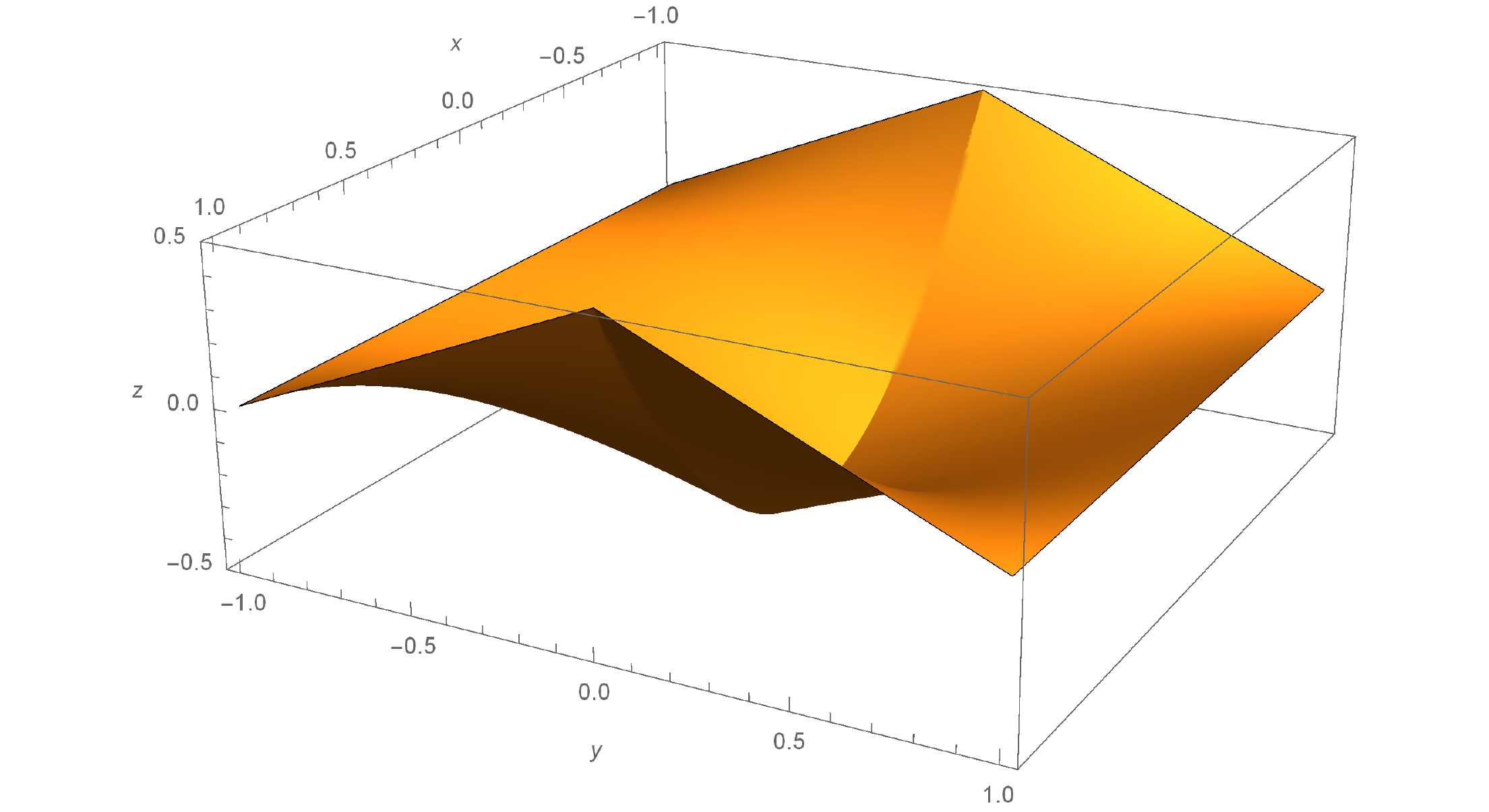} \includegraphics[width=0.32\textwidth]{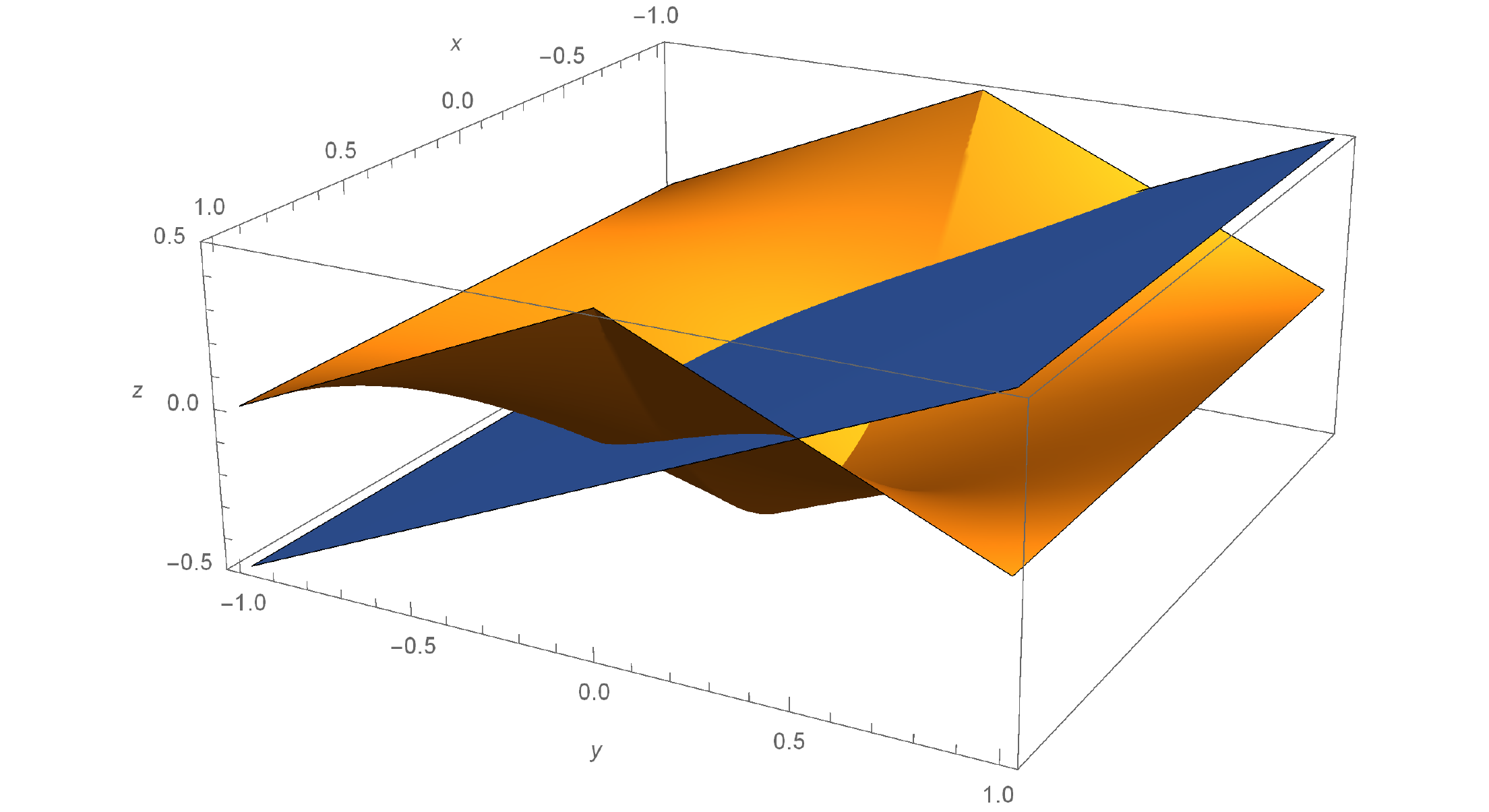} \includegraphics[width=0.32\textwidth]{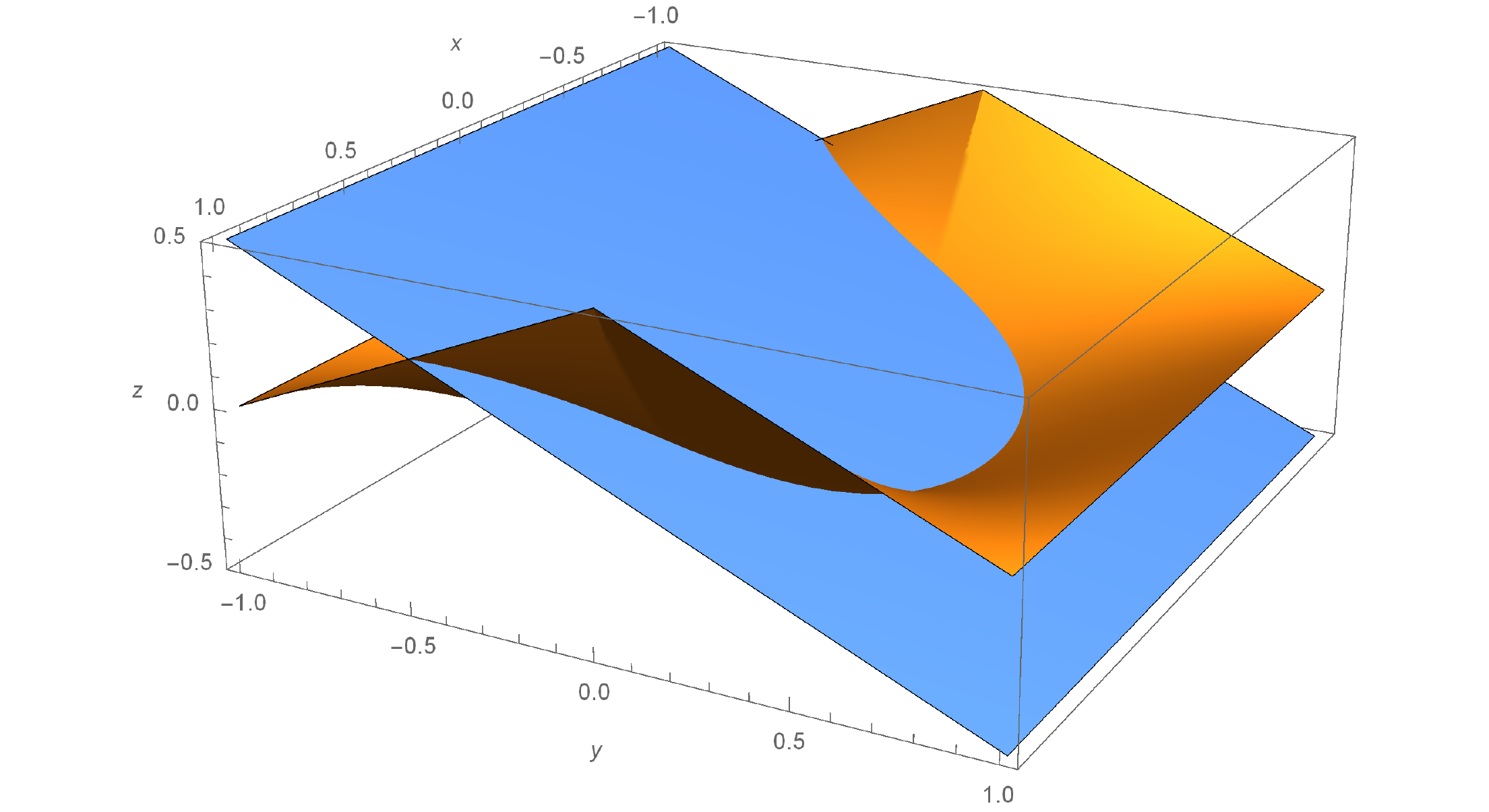}\\}
\caption{The function $h(x,y) = (x^2 - \frac{1}{2})(1 - |y|)$ on the left, and the same function shown together with two different best approximations: $q_{\frac{1}{2}}(x,y) = \frac{y}{2}$ and $q_{-\frac{1}{2}}(x,y) = -\frac{y}{2}$.}
\label{fig:function2}
\end{figure}


Finally, we would like to point out that smoothness of the function that we are approximating is not necessary for the uniqueness of a best approximation, as one may be tempted to conclude from the study of the functions $f$ and $h$. Note that for yet another modification,
\[
g(x,y) := (\min\{ |2x|, 2-|2x|\} - 1/2)(1-y^2),
\]
the function $q_0(x,y) \equiv 0$ is a unique best approximation, while the points of maximal and minimal deviation are distributed in a similar fashion, along the line $y=0$, potentially allowing for nonuniqueness. Notice that the function $g(x,y)$ is nondifferentiable at the points of minimal and maximal deviation. This function is however smooth in $y$ for every fixed $x$. This observation is related to the problem of relating the specific (partial) smoothness properties of the function we are approximating with the solution set. We discuss this open question in some detail in the conclusions section.
\end{example}

%

We have seen from the preceding example that whether the Chebyshev approximation problem has a solution is determined not only by the location of points of maximal and minimal deviation, but also by the properties of the function that is being approximated; in particular the smoothness of the function at the points of minimal and maximal deviation appears to be a decisive factor. 

\begin{example}\label{eg:bumps}For the distribution of points of maximal and minimal deviation from Example~\ref{eg:nonunique}, i.e. $N = \{z_1,z_3,z_5\}$, $P = \{z_2,z_4,z_6\}$, where $z_1,z_2,\dots z_6$ are defined by \eqref{eq:defz}, we construct a nonsmooth continuous function
\[
f(x) = \min\{ 2 \|x-z_1\|, 2 \|x-z_3\|, 2 \|x-z_5\|, 1\}-  \min\{  2 \|x-z_2\|,  2 \|x-z_4\|, 2 \|x-z_6\|, 1\},
\]
shown in Fig.~\ref{fig:nonsmoothbumpy} on the left.
\begin{figure}[ht]{\centering
\includegraphics[width=0.4\textwidth]{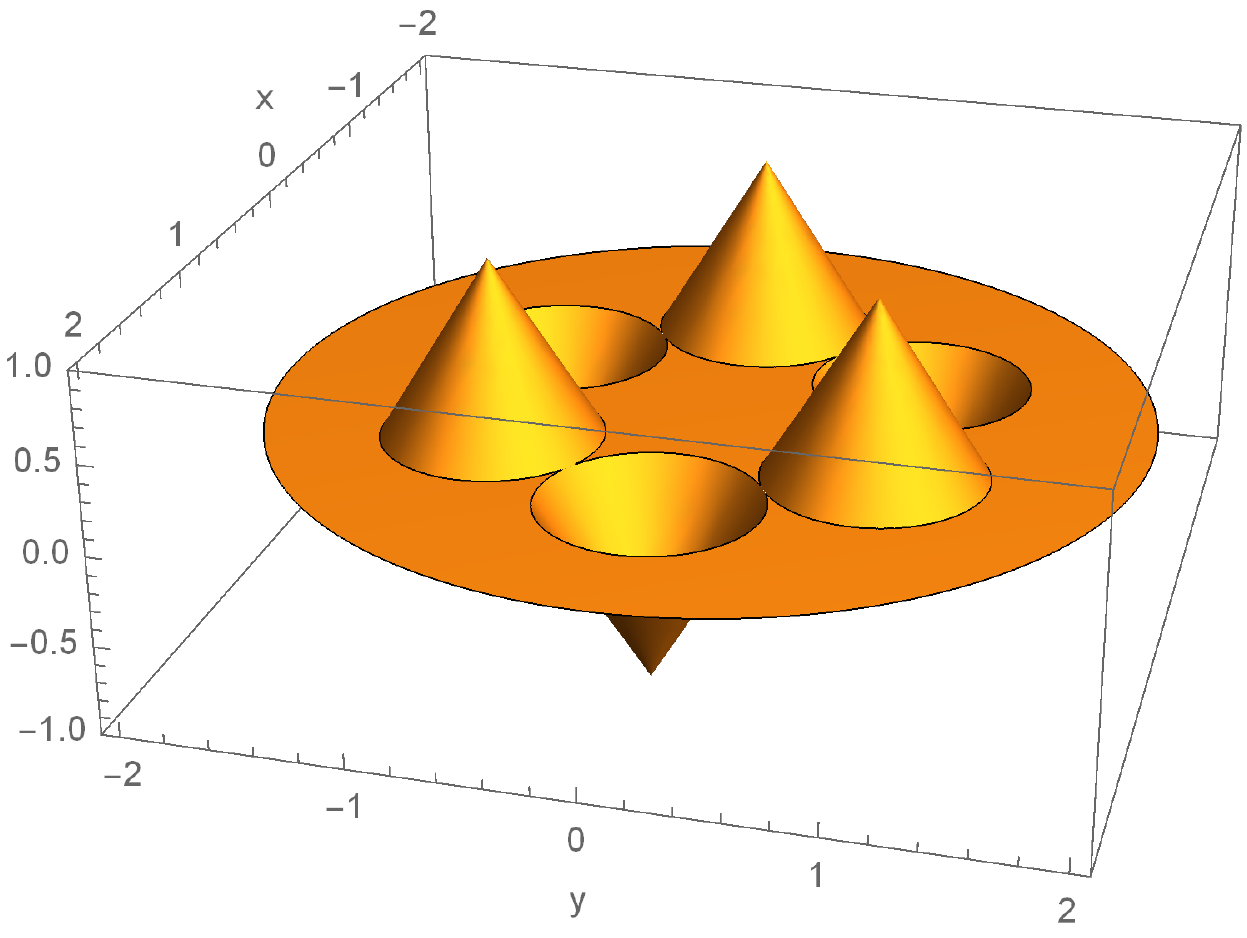}\qquad 
\includegraphics[width=0.4\textwidth]{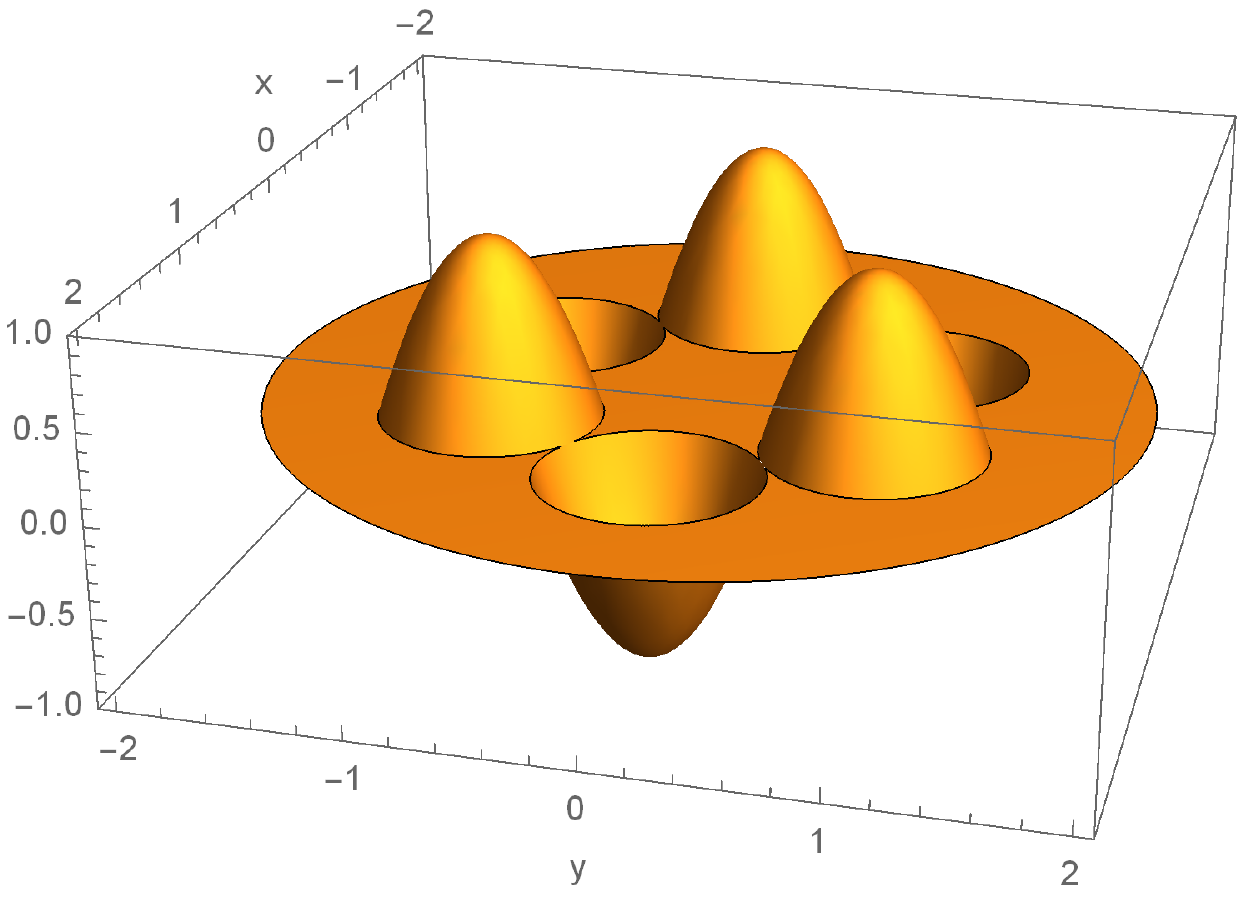}\\}
\caption{The functions $f$ and $h$ in Example~\ref{eg:bumps}.}
\label{fig:nonsmoothbumpy}
\end{figure}
The function $g(x,y) = 0$ is an optimal solution to the quadratic approximation problem for the function $f$ on $X = \{x\,|\, \|x\|\leq 2\}$ (since this is exactly the same pattern of points of minimal and maximal deviation as discussed in one of the two cases in Example~\ref{eg:nonunique}). Moreover, the polynomial 
\[
q_\alpha(x,y) = \alpha (x^2+y^2-1)
\]
is also a best approximation of $f$ for sufficiently small values of $\alpha$ (this may be already evident to the reader from the plot; the mathematically rigorous reasons for this will be laid out in the proof of Lemma~\ref{lem:sharpbump}).

Modifying the `bump' that defines each of the peaks that correspond to the points of minimal and maximal deviation so that the function $f$ smooth around these points, results in the uniqueness of the approximation $q_0$. Indeed, let 
\[
h(x) = \min\{ 4 \|x-z_1\|^2, 4 \|x-z_3\|^2, 4 \|x-z_5\|^2, 1\}-  \min\{  4 \|x-z_2\|^2,  4 \|x-z_4\|^2, 4 \|x-z_6\|^2, 1\},
\]
this function is shown in Fig.~\ref{fig:nonsmoothbumpy} on the right. 

The same constant polynomial $q_0(x,y) = 0$ is optimal for $h$, however, this time the solution is unique: indeed, suppose that another polynomial in $S$ provides a best approximation. This polynomial must be of the form $p_\alpha(x,y) = \alpha (x^2+y^2 -1)$ for some $\alpha \neq 0$. By convexity of the solution set, $p_{\alpha'}$ should also be optimal for any $\alpha'$ between 0 and $\alpha$. 

In the neighbourhood of the point $z_1$ we have $h(x,y) = 4 \|x-z_1\|^2-1 = 4\left[(x-1)^2 + y^2\right]-1$. Then for a sufficiently small $|\alpha'|$
\[
h\left(\frac{4}{4-\alpha'},0\right)-p_{\alpha'}\left(\frac{4}{4-\alpha'},0\right) = -1-\frac{(\alpha')^2}{4-\alpha'}<-1,
\]
hence this is not a solution. 
\end{example}

The next result provides a more general justification for the non-uniqueness of the approximation to a nonsmooth function $f$ that we have just considered. 

\begin{lemma}\label{lem:sharpbump} Let $V$ be as in \eqref{eq:V}, and let $N$ and $P$ be two disjoint compact subsets of $\R^n$ such that they can not be separated strictly by a polynomial in $V$.  Let 
\[
S = \{s\in V\,|\, s(x)\cdot s(y)\leq 0 \, \forall x\in P,y\in N\}. 
\]
There exists a continuous function $f:\R^n\to \R$ such that for any compact $X\in \R^n$ such that $N,P\subseteq X$, the optimal solution set $Q$ to the relevant optimisation problem satisfies $\dim Q = \dim S$, moreover, there exists $q_0\in Q$ such that $\P(q_0) = P$, $\N(q_0) = N$.
\end{lemma}
\begin{proof} Let 
\[
f(x):=\max_{u\in P}\varphi_u(x)-\max_{v\in N}\varphi_v(x),
\]
where 
\[
\varphi_u(x) = \max\left\{1-\frac{2}{d}\|x-u\|,0\right\}, \quad d= \min_{\substack{u\in P\\v\in N}}\|u-v\|.
\]
Fix a compact set $X\subset \R^n$ such that $P\cup N \subseteq X$. First observe that $q_0(x,y)\equiv 0$ is an optimal solution to the Chebyshev approximation problem: the deviation $f-q_0$ coincides with the function $f$, and  we have for all $x\in X$
\begin{align}\label{eq:001}
f(x) & = \max_{u\in P}\varphi_u(x)-\max_{v\in N}\varphi_v(x) \notag\\
& \leq  \max_{u\in P}\varphi_u(x)\notag\\
& =   \max_{u\in P}\max\left\{1-\frac{2}{d}\|x-u\|,0\right\}\notag \\
& \leq    \max_{u\in P}\max\left\{1-\frac{2}{d}\min_{u\in P}\|x-u\|,0\right\}\\
& =   \max\left\{1-\frac{2}{d}\min_{u\in P}\|x-u\|,0\right\}\notag\\
& = 1 -  \frac{2}{d} \min\left\{\min_{u\in P}\|x-u\|,\frac{d}{2}\right\}\leq 1;\notag
\end{align}
likewise 
\begin{align}\label{eq:002}
f(x) & \geq  -1 +  \frac{2}{d} \min\left\{\min_{v\in N}\|x-v\|,\frac{d}{2}\right\}\geq -1\quad \forall x\in X.
\end{align}
Moreover, for $x\in P$ we have $f(x) =1$, for $x\in N$ we have $f(x) = -1$, and it follows from \eqref{eq:001} and \eqref{eq:002} that for $x\notin P\cap U$ we have $-1<f(x)<1$, hence, $N= \N(q_0) $ and $P= \P(q_0)$, so $q_0$ satisfies the very last statement of the lemma. We have assumed that $N$ and $P$ can not be strictly separated by a polynomial in $V$, hence we deduce that $q_0\equiv 0$ is a best Chebyshev approximation of $f$ on $X$.

We will next show that for any direction $p\in S$ such that $p(N)\leq 0$ and $p(P)\geq 0$ there exists a sufficiently small $\alpha>0$ such that $\alpha p$ is another best Chebyshev approximation of $f$ on $X$. Note that this guarantees that for any set of linearly independent vectors in $S$ we can produce a simplex with vertices at zero and at nonzero vectors along these linearly independent vectors. This yields $\dim Q = \dim S$.

Since $p\in S$ is a polynomial, and the set $X$ is compact, $p$ is Lipschitz on $X$ with some constant $L$, and its absolute value is bounded by some $M> 0$ on $X$. Let 
$\alpha: = \min\left\{\frac 1M,\frac{2}{d L}\right\}$, then for $q = \alpha p$ we have
\[
|q(x)| = |\alpha p(x)|\leq \alpha \cdot M \leq 1 \; \forall x\in X, 
\]
\[
|q(x) -q(y)|  = |\alpha p(x) - \alpha p(y)| = \alpha | p(x) -  p(y)|\leq  \alpha L \|x-y\| \leq \frac{2}{d}\|x-y\|\quad \forall x,y \in X;
\]  
\[
q(y)- \frac{2}{d}\|x-y\|  \leq q(x) \leq  q(y)+\frac{2}{d}\|x-y\|\quad \forall x,y \in X;
\]  

From $q(N)\leq 0 $ and $p(P)\geq 0$ we have for all $x\in X$
\[
\frac{2}{d}\min_{y\in P} \|x-y\|  \leq \max_{y\in P} (q(y)- \frac{2}{d}\|x-y\| ) \leq q(x) \leq  \min_{y\in N} (q(y)+\frac{2}{d}\|x-y\|)\leq \frac{2}{d}\min_{y\in N} \|x-y\|.
\]  
Hence,
\[
\max\left \{\frac{2}{d}\min_{y\in P} \|x-y\| ,-1\right\} \leq  q(x) \leq \min\left \{ \frac{2}{d}\min_{y\in N} \|x-y\|,1\right\}.
\]

We hence have for every $x\in X$ 
\[
- 1 \leq f(x)-q(x)\leq   1 ,
\]
therefore $q$ is a best Chebyshev approximation of $f$ on $X$.
\end{proof}

Finally, we turn our attention to the relation between the uniqueness of best Chebyshev approximation and the geometry of the domain. We show that on finite domains the best approximation is nonunique whenever the dimension of $S$ allows for this (that is, $\dim S>0$ ).

\begin{lemma}\label{lem:finite} If $X\subset \R^n$ is finite, then for any $f:X\to \R$ we have $\dim Q =  \dim S$. 
\end{lemma}
\begin{proof} 
If $\dim S=0$, the result follows directly from Corollary~\ref{cor:unique}. For the rest of the proof, assume $\dim S>0$.

Let $q\in \relint Q$, $s\in S$. Then 
$$
s(x)\cdot s(y)\leq 0 \, \forall x\in \P,y\in \N.
$$ 
Let
$$
q_t: = q+ t s.
$$
Without loss of generality, assume that $s(x)\geq 0$ for $x\in \P$ and $s(x)\leq 0$ for $x\in \N$ (otherwise consider $-s$).

Let 
$$
\alpha:= 
\|f-q\|_{\infty}-\max_{x\in X\setminus (\N\cup \P)} |f(x)-q(x)|,
$$
where we use the standard convention that the maximum over an empty set equals $-\infty$, so $\alpha =+\infty$ in the case when $X=\N\cup \P$. Since $X$ is finite, $\alpha>0$.

Let
$$
\beta:= \max_{x\in X}|s(x)|.
$$

We have for all $t\geq 0$ and $x\in \N$
$$
\|f-q\|_\infty = q(x) - f(x) \geq q(x)-f(x)+ts(x)\geq q(x)-f(x) - t \beta = \|f-q\|_\infty - t \beta;
$$
for $t\geq 0$ and $x\in \P$
$$
\|f-q\|_\infty = f(x) - q(x) \geq f(x)-q(x)-ts(x)\geq f(x)-q(x) - t \beta = \|f-q\|_\infty - t \beta;
$$

For $x\in X\setminus (\N\cap \P)$ and all $t\geq 0$
$$ 
|f(x)-q_t(x)| \leq | f(x) -q(x)| + t|s(x)|\leq \|f-q\|_\infty-\alpha + t\beta. 
$$
Note that  $\alpha=+\infty$ only for the case when $X=\N\cap \P$. 

Therefore, for $t$ such that $t \beta \leq\min\{\alpha,\|f-q\|_\infty\}$ we have
$$
\|q_t -f\|_\infty \leq  \|f-q\|_\infty,
$$
and hence $q_t\in Q$ for some positive $t$.

It remains to pick a maximal linearly independent system $\{s_1,s_2,\dots, s_d\}\subset S$, and observe that $\conv\{q,q+t_1 s_1,\dots, q+ t_d s_d\}\subseteq Q$ for some nonzero $t_1,\dots, t_d$. Therefore, $\dim Q \geq \dim S$. By Lemma~\ref{lem:dim} the converse is true, and we are done.
\end{proof}

%
%
%
%
%
%
%

It follows from the previous lemma that the uniqueness of solutions depends not only on the function itself, but also on the domain of its definition. In particular, it may happen that a function defined on a continuous domain has a unique best approximation, but a discretisation of this domain would lead to nonuniqueness of best  approximation. This observation is crucial, since most numerical methods do require a certain level of discretisation.  In this case there is a potential danger of finding an optimal solution to the discretised problem, while it is not relevant to the original one.


\section{Conclusions}\label{sec:concl}

We have identified and discussed in detail key structural properties pertaining to the solution set of the multivariate Chebyshev approximation problem. We have clarified the relations between the points of maximal and minimal deviation   for different optimal solutions, related the set of optimal solutions to the set of separating polynomials, and elucidated the relations between the geometry of the domain and smoothness of the function and uniqueness of the solutions.

However many questions remain unanswered, some of them pertinent to the potential algorithmic solutions, and more remains to be done to fully understand the relation between the uniqueness of the solutions and structure of the problem. Namely, the following questions are of paramount importance.

\begin{enumerate}

\item Can we refine Mairhuber's theorem for the case of multivariate Chebyshev approximation by polynomials of degree at most $d$? Example~\ref{eg:nonunique} indicates that to have a unique approximation of any continuous function on a given domain by a system of multivariate polynomials, it may not be enough to restrict the domain to a set homeomorphic to a subset of a circle. Perhaps a more algebraic condition would work, for instance, restricting the domain to sets with one-dimensional Zariski closure.

\item What are the sufficient conditions for the uniqueness of the best Chebyshev approximation in terms of the function $f$ only? Can we guarantee that for a given set of points of maximal and minimal deviation there exists a domain $X$ that contains them and a function $f$ for which an optimal solution is unique and has specifically this distribution of points of minimal and maximal deviation?

\item Can we bridge the gap between Lemmas~\ref{lem:dim} and~\ref{lem:finite} and show that given a distribution of points of minimal and maximal deviation, for any $d\in \{0,\dots, \dim S\}$ there exists a function $f$ and domain $X$ with $\dim Q = d$? This question is closely related to our discussion at the end of Example~\ref{eg:unique-and-not}, where smoothness appears to be important only with relation to the orthogonal direction to the varieties separating the points of maximal and minimal deviation.

\end{enumerate}

\section*{Acknowledgements}

We are grateful to the Australian Research Council for supporting this work via Discovery Project DP180100602.

\bibliographystyle{plain}
\bibliography{refs}

\section*{appendix}

\subsection{Technical computations for Example~\ref{eg:nonunique}}  Consider the polynomial $q_\alpha(x,y)  =  3 \alpha (x^2+y^2 - 1)+1$, $\alpha\in [0,1]$, of which the polynomials $q_0$ and $q_1$ are special cases. Explicitly our deviation $d_\alpha(x,y) = f(x,y) - q_\alpha(x,y)$ has the form 
\[
d_\alpha(x,y) = x^6 +y^6+  3 x^4 y^2 + 3 x^2 y^4 + 6 x y^2 - 2 x^3-1- 3\alpha(x^2+y^2-1).
\]
The points of maximal and minimal deviation are the global extrema of $d_\alpha$ on the unit disk. To obtain all such extrema, we first find the global minima and maxima of $d_\alpha$ on the boundary of the disk, using the method of Lagrange multipliers, and then study the behaviour of $d_\alpha$ on the interior of the disk.

Our Lagrangian function is $L_\alpha(x,y,\lambda) = d_\alpha(x,y) + 6\lambda (x^2+y^2-1)$ (where we have multiplied the constraint by 6 for convenience), and the necessary condition for the constrained global stationary points on the unit circle is
\begin{align*}
\nabla L_\alpha(x,y,\lambda) 
 &  = 6\begin{bmatrix} ( y^2 -x^2)+    x [(x^2+  y^2)^2-\alpha+ 2\lambda]\\
  2 xy  +  y[(x^2+  y^2)^2-\alpha+ 2\lambda]
  \\
   x^2+y^2 -1 
  \end{bmatrix} = 0_3.
\end{align*}

Multiplying the first line by $y$, and the second line by $x$, and subtracting, we obtain the consequence of the first two equations in the Lagrangian system: $y ( y^2 - 3 x^2) = 0$. Together with the constraint $x^2= 1-y^2$ this yields six candidates for the stationary points on the boundary,  
\[
z_1 =  (1,0),\; z_2 = \left(\frac{1}{2},\frac{\sqrt{3}}{2}\right),\;  z_3 = \left(-\frac{1}{2},\frac{\sqrt{3}}{2}\right),
\]
\[
 z_4 = (-1,0), \; z_5 =  \left(-\frac{1}{2},-\frac{\sqrt{3}}{2}\right),\; z_6 = \left(\frac{1}{2},-\frac{\sqrt{3}}{2}\right).
\]
It is not difficult to check that 
\[
d_\alpha (z_1) = d_\alpha(z_3) = d_\alpha(z_5) = -2, \quad 
d_\alpha (z_2) = d_\alpha(z_4) = d_\alpha(z_6) = 2. 
\]
Note that these values do not depend on $\alpha$. 

It remains to study the behaviour of the deviation $d_\alpha$ on the interior of the disk. If $d_\alpha$ attains a global minimum or maximum in an interior point of the disk, then such extrema must satisfy the unconstrained optimality condition $\nabla d_\alpha(x,y) = 0_2$. We have explicitly 
\begin{align*}
\nabla d_\alpha(x,y) &=  6\begin{bmatrix} ( y^2 -x^2)+    x [(x^2+  y^2)^2-\alpha]\\
  2 xy  +  y[(x^2+  y^2)^2-\alpha]
  \end{bmatrix} = 0_2.
\end{align*}
As before, premultiplying the equations by $y$ and $x$ and subtracting, we conclude that any stationary point must satisfy the equality  $y ( y^2 - 3 x^2) = 0$. Hence any maximum or minimum must lie on one of the lines 
\[
y=0, \quad y = - \sqrt{3} x, \quad y = \sqrt{3}x.
\]
Observe that both our polynomial and the constraint are symmetric with respect to the rotation of the plane by $2\pi/3$, the restrictions of the polynomial $d_\alpha$ to each of those lines are identical (under the relevant rotations), hence it is sufficient for us to study the behaviour of the restriction of  $d_\alpha$ to the open line segment $(-1,1)\times\{0\}$. For convenience, we let
\[
\varphi_\alpha(x) : = d_\alpha(x,y) = x^6- 2 x^3 - 1 - 3 \alpha (x^2 -1).
\]
Observe that
\[
\varphi'_\alpha(x)  = 6 x^5- 6 x^2 - 6 \alpha x = 6x(x(x^3 - 1) -1)<0 \quad \forall x\in (0,1),
\]
hence $\varphi_\alpha(x)$ is strictly decreasing on $(0,1)$, and can only have  minima or maxima on the endpoints of  $[0,1]$.
For the open line segment $(-1,0)$ and $\alpha \in [0,1]$ we have 
\[
\varphi_\alpha(x) 
= (x^3 -1)^2 + 3\alpha (1-x^2)-2 < 1+ 3 - 2 = 2\quad \forall  x\in (-1,0);  
\]
likewise
\[
\varphi_\alpha(x)  =  (x^3 -1)^2 + 3\alpha (1-x^2)-2 > 0 + 0 -2  = -2 \quad \forall  x\in (-1,0).  
\]
Since $\varphi(1) = -2$, and $\varphi(1) = 2$, this means that no global minimum or maximum can be achieved on $(-1,1)$. We are hence left with the only candidate $x=0$, for which we have 
\[
\varphi_\alpha(0) = -1 + 3\alpha \in [-1,2) \quad \text{for } \alpha \in [0,1),
\]
and $\varphi_1(0) = -2$. 
This yields the distribution of points of minimal and maximal deviation of $f$ from $q_0$ and $q_1$ as described in Example~\ref{eg:nonunique}.

\subsection{Computations for Example~\ref{eg:unique-and-not} }

To find the points of maximal and minimal deviation of $f(x,y) = (x^2 - \frac{1}{2})(1 - y^2)$ from the constant polynomial $q_0(x,y)\equiv 0$ on the square $[-1,1]\times [-1,1]$, observe that the optimality condition on the interior of the square gives
\[
\nabla f(x,y) = \begin{pmatrix}
2 x (1-y^2)\\
 - y\left(x^2-\frac{1}{2}\right) 
 \end{pmatrix} = 0_2,
 \] and  out of the five solutions to $\nabla f(x,y) = 0$ 
\[
(0,0), \; \left(-\frac{1}{\sqrt 2},-1\right), \; \left(\frac{1}{\sqrt 2},-1\right),\; \left(-\frac{1}{\sqrt 2},1\right),\; \left(\frac{1}{\sqrt 2},1\right)
\] 
only $(0,0)$ is in the interior of the square. Hence we have only one stationary point $(0,0)$ within the interior of the square, with deviation $d_0 (0,0) = f(0,0)- q_0(0,0) = f(0,0) = -\frac{1}{2}$. 

We now study the boundary of the square: restricting to $x=\pm 1$, and $y\in [-1,1]$, we have the function $\frac{1}{2}(1-y^2)$, which attains minima at the endpoints of the sides of the square, at $(\pm 1,\pm 1)$ with deviation $d_0(\pm 1,\pm 1) = 0$,  and maxima at $(\pm 1,0)$, with the value $d_0(\pm 1,0) = f(\pm 1,0) = \frac{1}{2}$. For $y=\pm 1$ the function is identically zero. We conclude that the points of maximal and minimal deviation of $f$ from zero, on the square $X = [-1,1]\times [-1,1]$, are 
\[
\P(q) = \{(0,0)\}, \qquad \N(q) = \{(-1,0), (1,0)\}.
\]

We next study the deviation of the function $h(x,y) = (x^2 - \frac{1}{2})(1 - |y|)$ from polynomials $q_\alpha(x,y) = \alpha y$ for $\alpha \in \left[-\frac{1}{2}, \frac{1}{2}\right]$. First of all, observe that for $y =0$ we have 
\[
d_\alpha(x,y) = h(x,y) - q_\alpha(x,y) =x^2 - \frac{1}{2},
\]
and hence $d_\alpha(x,0)$ is minimal at $(0,0)$ with the value $d_\alpha(0,0) = -\frac{1}{2}$, and maximal at $(\pm 1,0)$ with the value $d_\alpha(\pm 1,0)= \frac{1}{2}$, independent on $\alpha$. 

For $y>0$ we have $d_\alpha(x,y) = h(x,y) - q_\alpha(x,y)  = (x^2 - \frac{1}{2})(1 - y)- \alpha y $, hence the unconstrained optimality condition gives
\[
\nabla d_\alpha(x,y) =  \begin{pmatrix}
2 x (1-y)\\
 - x^2 +\frac 12 - \alpha 
 \end{pmatrix} = 0_2,
 \] and the only case when we have solutions in the intersection of the interior of the square and $y>0$ is when $\alpha = \frac 1 2$; likewise,  $\nabla d_\alpha(x,y)=0$ gives no solutions in the interior of the square intersected with $y<0$ except for $\alpha = -\frac{1}{2}$. In both cases we have 
\[
d_{\frac 1 2}(0,y) = -\frac 1 2 (1-y) - \frac{1}{2} y= - \frac{1}{2}, \quad y >0; 
\]
\[
d_{-\frac 1 2}(0,y) = -\frac 1 2 (1+y) + \frac{1}{2} y= - \frac{1}{2}, \quad y <0. 
\] 

For the sides of the square that correspond to  $x=\pm 1$, and $y\in [-1,1]$, we have a piecewise linear function
\[
d_\alpha(\pm 1,y) =  \frac{1}{2}(1 - |y|)- \alpha y = 
\begin{cases}
\frac{1}{2}- (\alpha + \frac 1 2)  y, & y\geq 0,\\ 
\frac{1}{2}- (\alpha - \frac 1 2)  y, & y< 0, 
\end{cases}
\]
hence its behaviour is completely determined by the endpoints of the relevant segments: $(\pm 1,\pm 1)$, $(\pm 1,0)$. We have 
\begin{equation}\label{eq:additional}
d_\alpha(\pm 1,-1) = \alpha;  \quad d_\alpha(\pm 1,1) = -\alpha;\quad d_\alpha(\pm 1,0) = \frac{1}{2}.
\end{equation}
For the remaining case of the interior of the sides, $(-1,1)\times \{\pm 1\}$ we have 
\begin{equation}\label{eq:endpoints}
d_\alpha(x,-1) = \alpha; \quad d_\alpha(x,1) = - \alpha. 
\end{equation}

Observe that for $\alpha =0$ the only points of maximal and minimal deviation lie on the line $y=0$, and hence the polynomial $q_0(x,y) = 0$ is a best approximation of the function $h$ on the square $X$.  Also note that for $|\alpha|> \frac 1 2$ the relations \eqref{eq:endpoints} give worse values of minimal and maximal deviation, hence, $q_\alpha$ can not be a best approximation for $|\alpha |>\frac 1 2 $. For $|\alpha| \in (0,\frac 1 2 ) $ we observe that there are no additional points of minimal and maximal deviation on top of the three alternating points on $y=0$ that are present for $\alpha = 0$. It remains to consider the values $|\alpha| = \frac 1 2 $. 

For $\alpha = -\frac 1 2$ we have from \eqref{eq:additional} and the piecewise linear observation 
\[
d_{-\frac 1 2}(\pm 1,y) = \frac{1}{2} \qquad \forall y \in [0,1],
\] 
and \eqref{eq:endpoints} gives
\[
d_{-\frac 1 2}(x,-1) = -\frac 1 2; \quad d_{- \frac 1 2 }(x,1) = \frac 1 2 . 
\]
Likewise, for $\alpha = \frac 1 2 $ we obtain 
\[
d_{\frac 1 2}(\pm 1,y) = \frac{1}{2} \qquad \forall y \in [-1,0],
\qquad 
d_{\frac 1 2}(x,-1) = \frac 1 2; \quad d_{\frac 1 2 }(x,1) = -\frac 1 2 . 
\]
 	
\end{document}